\documentclass[10pt,leqno]{article}

\usepackage[utf8]{inputenc}

\usepackage[T1]{fontenc}

\usepackage{amsmath,amssymb,amsthm, epsfig}

\usepackage{hyperref}

\usepackage{amsmath}

\usepackage{amssymb}

\usepackage{hyperref}

\hypersetup{pdfborder = {0 0 0 0}}

\usepackage{color}

\usepackage{ulem}

\usepackage{epsfig}

\title{Eigenvalues of the Neumann Laplacian in symmetric regions}
\author{ Marcus A. M. Marrocos\footnote{Partially supported by CAPES-Brazil and Fapeam-Brazil} 
\quad Ant\^onio Luiz Pereira\footnote{Partially supported by FAPESP-Brazil 2008/55516-3} }

\newlength{\hchng}
\newlength{\vchng}
\newcommand{\front}[1]{\partial \Omega_{#1}}
\newcommand{\derivada}[3]{\dfrac{\partial^{#1}{#2}}{\partial{#3}^{#1}}}
\newcommand{\nablafront}[1]{\nabla_{\partial\Omega_{#1}}}
\newcommand{\divfront}[1]{div_{\partial{\Omega_{#1}}}}
\newcommand{\umpt}[1]{\stackrel{\cdot}{{#1}}}
\newcommand{\dpt}[1]{\stackrel{\cdot\cdot}{{#1}}}
\newcommand{\R}{\mathbb{R}}

\newtheorem{teo}{Theorem}
\newtheorem{lema}{Lemma}
\newtheorem{coro}{Corollary}
\newtheorem{defi}{Definition}
\newtheorem{conjec}{Conjecture}
\newtheorem{obs}{Remark}

\newtheorem{prop}{Proposition}

\numberwithin{equation}{section}

\newcommand{\intav}[1]{\mathchoice {\mathop{\vrule width 6pt height 3 pt depth  -2.5pt
\kern -8pt \intop}\nolimits_{\kern -6pt#1}} {\mathop{\vrule width
5pt height 3  pt depth -2.6pt \kern -6pt \intop}\nolimits_{#1}}
{\mathop{\vrule width 5pt height 3 pt depth -2.6pt \kern -6pt
\intop}\nolimits_{#1}} {\mathop{\vrule width 5pt height 3 pt depth
-2.6pt \kern -6pt \intop}\nolimits_{#1}}}

\begin{document}
\maketitle

\begin{abstract}
 In this work we are concerned with the multiplicity of the eigenvalues of the Neumann Laplacian in regions of $\mathbb{R}^n$  which are invariant under the natural action of a compact subgroup $G$ of $O(n)$. We   give a partial positive answer (in the Neumann case)  to a conjecture of V. Arnold  \cite{arnold} on the transversality of the transformation given by the Dirichlet integral to  the stratification in the space of quadratic forms according to the multiplicities of the eigenvalues.
 We show, for some classes of subgroups  of $O(N)$ that, generically in the  set of  $G-$invariant,  $\mathcal{C}^2$-regions, the action is irreducible  in each eigenspace   $Ker(\Delta+\lambda)$. These classes include finite subgroups with irreducible representations of dimension not greater than 2 and, in the case  $n=2$,  any compact subgroup of $O(2)$.   We also obtain some partial results for general compact subgroups of $O(n)$.  

\medskip


\medskip

%\noindent \textbf{AMS Subject Classifications:} 35J60, 35J75,

\end{abstract}

\tableofcontents

\section{Introduction}

 Perturbation of the boundary in boundary value problems have been considered by many authors, from various points of view since the classical works of  J. Hadamard \cite{hadamard} and  J.W.S. Rayleigh \cite{rayleigh}.  We also mention the more recent works \cite{hp,tl,rousselet,simon,sokolowski}. In particular, generic properties for the solutions of boundary value problems have been proved  in    \cite{micheletti-eliptico,micheletti-laplace,teman}.

In \cite{hp}, D. Henry developed a kind of differential calculus where the independent variable is the domain of definition of the differential equation.  In this way, he was able to use standard analytic  tools  such as the Implicit Function Theorem and the Lyapunov-Schimdt method. In the same work, he proved a generalized version of the Transversality Theorem of Thom and Abraham and applied it to obtain generic properties  for the solution of elliptic equations with various boundary conditions.  

 Generic properties for the eigenvalues and eigenfunctions of elliptic problems have  also been investigated by many authors,  among which  we mention   \cite{zp,micheletti-laplace,as,ta,msb,z3,uhlenbeck}. The generic situation for the eigenvalues of elliptic problems in symmetric regions has been specifically considered in  \cite{zp,as,ta,msb,z3}. 

 One can find at least two approaches in the literature to deal with the problem of simplicity of the eigenvalues for elliptic problems: using the expression of the derivatives of the eigenvalues as functions of the domain or the Transversality Theorem. The first method is used, for instance, in   \cite{zp,hp,micheletti-laplace,z3}. A combinations of the two methods is  used in  \cite{as,ta,msb,uhlenbeck}.

If $G$ is a compact subgroup of $O(n)$, we say that a region $\Omega \subset 
 \R^n$ is  \textit{$G$-symmetric} if it is invariant under the natural action of
 $G$. 
 In \cite{arnold}, V. Arnold conjectured  that the transversality  of the transformation given by the Dirichlet integral to  the stratification in the space of quadratic forms according to the multiplicities of the eigenvalues should  be the generic situation for  the eigenvalues of the Dirichlet Laplacian in symmetric regions.
 Equivalently, in the generic situation, the representation  $\Gamma:G\rightarrow L^2(\Omega)$ given by  $\Gamma_gu=u\circ g^{-1}$  should be irreducible   in the set of regular  bounded $G$-symmetric regions, when restricted to  the eigenspaces  of the Neumann Laplacian. 

 The first partial answer to Arnold's conjecture was given 
 in  \cite{z3}, for  $\mathbb{Z}_3$ symmetric regions.
 In this particular case, there are only two possibilities for the eigenfunctions, they are either symmetric:  $u\circ g^{-1}=u$, or ``anti-symmetric'': $u+u\circ g^{-1}+u\circ (g^2)^{-1}=0$, where  $g\in O(n)$  is a generator of  $\mathbb{Z}_3$.
 Theorem  1.1 of  \cite{z3} states that, generically in the set of 
  $\mathbb{Z}_3$ symmetric regions, the symmetric eigenvalues (that is, whose associated  eigenfunctions are all  symmetric) of the Dirichlet Laplacian are
 all simple, and the ``anti-symmetric'' eigenvalues are all double. However, the author does not take into account the possibility of the existence of eigenvalues with both symmetric and ``anti-symmetric'' eigenfunctions.   

 The complete answer to the question of the  genericity of the eigenvalues of the Dirichlet Laplacian in planar  $\mathbb{Z}_3$-symmetric regions was given  in
 \cite{zp}.  In the same work, the author also considered planar regions with
 $\mathbb{Z}_p$ symmetry for  $p=2,3,4$.

  A detailed investigation of   the generic situation of the eigenvalues of the Dirichlet Laplacian in symmetric regions is done in  \cite{ta} or  \cite{as}. In particular,  conditions  for the existence of multiple eigenvalues on  $G$-symmetric  region  are established for arbitrary compact subgroups of $O(n)$. More precisely, it is shown there  that, if   $G< O(n)$ 
 is compact and  $\Omega$ has a free point under the action  $G$,
 then there always exist multiple eigenvalues, except in the exceptional case 
  $G = \mathbb{Z}_2\oplus...\oplus\mathbb{Z}_2$, (see corollary \ref{cam}).
 The presence of a free point under the action  $G$ guarantees the existence of 
 irreducible sub-representations of  $\Gamma$ for each possible class. As a consequence,  it follows that for each irreducible representation of  $\Gamma$
 there exists an eigenvalue with multiplicity at least equal to the dimension of the sub-representation  (see theorem \ref{tam}).
 Therefore, the best we can hope for is for  the sub-representation 
  $\Gamma|_{ker(\Delta+\lambda)}$ to be 
 irreducible for any eigenvalue  $\lambda$ in
 a generic set of bounded regular $G$-symmetric regions of  $\mathbb{R}^n$.
  
 Indeed, it is shown in \cite{as} that this is true for some classes of finite groups, namely commutative groups and non commutative groups whose irreducible representations have at most dimension 2 (see theorem 7.1 of \cite{as}).
 Though not explicitly stated in \cite{as},  the genericity property follows then for planar regions and arbitrary subgroups of $O(2)$ (see  remark
\ref{planarsimple}). 

 In \cite{as}, \cite{msb} 
  the theory developed by Henry in  \cite{hp}  is also used to prove some generic properties for the eigenvalues of the Dirichlet Laplacian and Bilaplacian on
 symmetric domains, using Henry's Transversality theorem as the main tool. 

 Here, we obtain some partial answers to the Arnold's conjecture for
 the Neumann Laplacian on symmetric regions. More precisely, we consider the problem

  \begin{equation}\label{eln}
                 \left\{
                       \begin{array}{lccc}
                             (\Delta+ \lambda)u=0,&   \ in \ \  \Omega;\\
                              \frac{\partial u}{\partial N}=0, & \ \ on\ \ \partial \Omega.\\
                              
                       \end{array}
                            \right.
\end{equation}                
 Following the formulation of \cite{as}, we call an eigenvalue 
\textit{ $G$-simple} if  the action  $\Gamma|_{ker(\Delta+\lambda)}$ is irreducible and investigate the validity of the following

\begin{conjec}\label{conject1}
Let  $G$ be a compact subgroup of  $O(n)$. Then, in a residual set of bounded, regular $G$-symmetric  regions  of  $\mathbb{R}^n$  the eigenvalues of the Neumann Laplacian are all  $G$-simple. 
\end{conjec}
 
 The representation $\Gamma$, which will be called here the \textit{quasi-regular representation} of  $G$  in  $L^2(\Omega)$, induces an orthogonal decomposition in the space   $L^2(\Omega)$ (see theorem  \ref{td}), that is  
$$
L^2(\Omega)=\displaystyle\bigoplus_{\sigma\in\hat{G}}M_{\sigma},
$$
where each subspace  $M_{\sigma}$ is invariant by the Laplacian operator (see proposition \ref{pi}). These spaces will be called \textit{symmetry spaces}.

%%%%%%%%%%%%%%%%%%%%%%%%%%%%%%%%%%
%%%%%%%%%%%%%%%%%%%%%%%%%%%%%%%%%%

%Supondo que $\Omega$ possua um ponto livre, A.L.Pereira \cite{as}, \cite{ta} mostrou que em regiões $\Omega$ $G$-simétricas sempre existem autovalores múltiplos para o Laplaciano, exceto para subgrupos $G$ isomorfos a $\mathbb{Z}_2\oplus...\oplus\mathbb{Z}_2$ ($m$ vezes). Mais precisamente, os autovalores possuem multiplicidade tão grande quanto forem as dimensões das representações irredutíveis de $G$.

%%%%%%%%%%%%%%%%%%%%%%%%%%%%%%%%%%
%%%%%%%%%%%%%%%%%%%%%%%%%%%%%%%%%%

 The conjecture \ref{conject1}  can be split in two sub-conjectures: 

\begin{description}
	\item{(I)}  In a residual set of  $G$-symmetric regions of
 $\mathbb{R}^n$, the representation $\Gamma$ of  $G$ in  $Ker(\Delta+\lambda)\cap M_{\sigma}$  is irreducible, for each  $\sigma\in \hat{G}$.
	\item{(II)}  In a residual set of  $G$-symmetric regions of
 $\mathbb{R}^n$, there are no eigenvalues with eigenfunctions belonging to two different  symmetry spaces.  
\end{description}

 In fact  we analyze here the validity of conjecture \ref{conject1} only for
\textit{finite groups}. The case of  infinite groups presents  additional technical difficulties and will be consider in a forthcoming paper.

 In what follows, we will say that an eigenvalue  $\lambda$ of the Laplacian restricted   to   $M_{\sigma}$  is  $G_{\sigma}$-simple if the quasi-regular representation $\Gamma$ of $G$ in  $Ker(\Delta+\lambda)\cap M_{\sigma}$ is irreducible.

 Theorem 1 of \cite{z3} proves then that, generically in the set of bounded   $\mathbb{Z}_3$-symmetric regions of  $\mathbb{R}^n$ the eigenvalues of the Dirichlet Laplacian are all  $G_{\sigma}$-simple. 
 
 We   show the validity of sub-conjecture I, for any finite subgroup of   $O(n)$ (see corollary \ref{cp1})
 that is all eigenvalues of the Neumann Laplacian are  $G_{\sigma}$-simple.
 % This result represents a major advancement in an attempt to get the proof of the conjecture considering %any compact groups.
%  We also considered infinite subgroups but in that case an extra technical condition ( $dim G<n-1$) (see \ref{cp1i}).
 %This result is an important first step to the proof of the whole conjecture \ref{conject1} for infinite subgroups.
 The main result of this work is  that \ref{conject1} is true for finite subgroups with irreducible representations of dimension at most 2. As a corollary, we obtain a proof of the conjecture for arbitrary subgroups of   $O(2)$  in planar regions.

\section{Preliminaries}\label{Pre}

%%%%%%%%%%%%%%%%%%%%%%%%%%%%%%%%%%%%%%%%%%%%%%%%%%%
%%%%%%%%%%%%%%%%%%%%%%%%%%%%%%%%%%%%%%%%%%%%%%%%%%%

  In this section we present some results  on boundary perturbations that will be needed in the sequel. More details and proofs can be found in \cite{hp}.
 
\subsection{Definitions and preliminary results}
 
% We first fix some notation and give some definitions to be used along the work.

We represent a point  $x\in\mathbb{R}^n$ as a  $n$-uple of real numbers   $x=(x_1,...,x_n)$ and use the multi-index notation for the partial derivatives.
$$
\partial_x^{\alpha}=\left(\frac{\partial}{\partial x}\right)^{\alpha}=\frac{\partial^{\alpha_1}}{\partial x_1^{\alpha_1}}\frac{\partial^{\alpha_2}}{\partial x_2^{\alpha_2}}...\frac{\partial^{\alpha_n}}{\partial x_n^{\alpha_n}}
$$
where $\alpha=(\alpha_1,...\alpha_n)\in\mathbb{N}$, $|\alpha|=\alpha_1+\alpha_2+...+\alpha_n$.  Partial derivatives will also be denoted by 
$$
D_i=\frac{\partial}{\partial x_i}\,\, e \,\, D^{\alpha}=D^{\alpha_1}_1...D^{\alpha_n}_n
$$

 If  $f:\mathbb{R}^n\rightarrow\mathbb{R}$ is   $m$-times differentiable at a point  $x$, its  $m$-th derivative may be considered
 %  either as   the collection of the partial derivatives of order $m$
% $$
% D^mf(x)=\left\{\left(\frac{\partial}{\partial x}\right)^{\alpha}; |\alpha|=m \right\}
% $$
% or 
as  a $m$-linear symmetric form in  $\mathbb{R}^n$
$$
h\mapsto D^mf(x)h^m
$$
with norm
$$
|D^mf(x)|=\displaystyle\max_{|h|\leq1}|D^mf(x)h^m|.
$$

 We denote the boundary of an open subset $\Omega$ of  $\mathbb{R}^n$  by $\partial\Omega$ and its closure by  $\overline{\Omega}$. Given a normed vector space  $E$ we denote by  $\mathcal{C}^m(\Omega,E)$ the space of   $m$-times continuously and bounded differentiable functions  $f:\Omega\rightarrow E$ whose derivatives extend continuously to $\overline{\Omega}$, with norm
$$
||f||_{\mathcal{C}^m(\Omega,E)}= \displaystyle\max_{0\leq j\leq m}\sup_{x\in\Omega}|D^jf(x)|.
$$
If $E=\mathbb{R}$, we denote  $\mathcal{C}^m(\Omega,E)$ simply by  $\mathcal{C}^m(\Omega)$. We also define the subspaces
\begin{itemize}
\item $\mathcal{C}^m_0(\Omega,E)$,  the subspace  of $m$-th continuously differentiable functions with compact support in $\Omega$.
\item $\mathcal{C}^m_{unif}(\Omega,E)$ is the closed subspace of functions in  $\mathcal{C}^m(\Omega,E)$ with $m$-th derivative
 uniformly continuous. 
\item $\mathcal{C}^{m,\alpha}(\Omega,E)$  is the closed subspace of functions in  $\mathcal{C}^m(\Omega,E)$ with  H\"{o}lder continuous $m$-th derivative
 and norm
$$
||f||_{\mathcal{C}^{m,\alpha}(\Omega,E)}=\max\left\{||f||_{\mathcal{C}^m(\Omega,E)}, H^{\Omega}_{\alpha}(D^mf)\right\}
$$
where 
$$
 H^{\Omega}_{\alpha}(f)=\sup\left\{\frac{|f(x)-f(y)|}{|x-y|^{\alpha}}; x\neq y\in\Omega\right\}.
$$
\end{itemize}

We say that an open set   $\Omega\subset\mathbb{R}^n$  is  $\mathcal{C}^m$-regular  or has  $\mathcal{C}^m$-regular boundary if there exists  $\phi\in \mathcal{C}^m(\mathbb{R}^n,\mathbb{R})$, $m\geq 2$  or at least $\mathcal{C}^1_{unif}$, such that 
$$
\Omega=\{x;\phi(x)>0\}
$$
and  $\phi(x)=0$ implies $|\nabla\phi(x)|\geq1$. 

 It is proved in \cite{hp} that, for bounded open sets, the above definition is equivalent to the ones in
  \cite{adn} and  \cite{eb}. 

Besides these spaces of smooth functions, we will frequently work on Sobolev spaces, of which we present some basic definitions below.

Let  $m$ be a non negative integer,  $1\leq p<\infty$ and $\Omega \subset \R^n$ an open bounded set.
 If  $u\in \mathcal{C}^m(\Omega)$ we define the norm
$$
||u||=\left(\int_{\Omega}\sum_{|\alpha|\leq m}|D^{\alpha}u|^{p}dx\right)^{\frac{1}{p}}.
$$
The completion of  $\mathcal{C}^m(\Omega)$ with respect to this norm  is denoted by  $H^{m,p}(\Omega)$.
We also consider  $W^{m,p}(\Omega)$, the space of functions  $m$-th weakly differentiable, whose
 weak derivatives up to order  $m$ belong to  $L^p(\Omega)$. It can be proved that  $W^{p,m}(\Omega)=\mathcal{C}^{p,m}(\Omega)$ when $\Omega$ is  $\mathcal{C}^m$-regular. 
If  $p=2$, we use the notation $H^{m,p}(\Omega)=H^m(\Omega)$.

We also define  $H_0^{m,p}(\Omega)$ as the completion of  $\mathcal{C}^m_0(\Omega)$ and $W_0^{m,p}(\Omega)$ the space of functions in $W^{m,p}(\Omega)$ satisfying $D^{\alpha}u=0$ on $\partial\Omega$ for  $|\alpha|\leq\frac{m}{2}$.

For functions  $\phi$ defined in $\partial \Omega$, we can introduce the class of functions
  $W^{m-\frac{1}{p},p}(\partial\Omega)$  in  such a way that   
  $\phi \in W^{m-\frac{1}{p},p}(\partial\Omega)$  if and only if it is the boundary value of 
 functions in   $v\in W^m(\Omega)$ with norm 
$$
||\phi||=\inf||v||_{W^{m,p}(\Omega)}
$$
where the infimum is taken over all
 $v\in W^m(\Omega)$ such that $v_{|_{\partial\Omega}}=\phi$, where $v_{|\partial\Omega}$ is the trace of $v$ on $\partial\Omega$  (see \cite{lions}).  
%(DEF. ESTRANHA, SERIA PRECISO PROVAR QUE ESTA BEM DEFINIDA. FUNÇÕES EM Lp EM GERAL SÓ ESTÃO DEFINIDAS Q.S.- DAR REFERÊNCIA?). 

We also frequently encounter differential operators on hypersurfaces of
 $\mathbb{R}^n$. 
 
Let  $S$ be a $\mathcal{C}^1$ hypersurface in  $\mathbb{R}^n$ and  $\phi: S\rightarrow\mathbb{R}$ a $\mathcal{C}^1$ functions. The \textit{tangential gradient }
of $\phi$ is the tangent vector field in $S$ such that, for any
  (sufficiently smooth) curve  $x(t)$ in  $S$, we have
$$
\frac{d}{dt}\phi(x(t))=\nabla_{S}\phi(x(t))\cdot\dot{x}(t).
$$

 If  $S$ is of class  $\mathcal{C}^2$ and  $\stackrel{\rightarrow}{a}$   is a  $\mathcal{C}^1$ vector field on  $S$, we define its \textit{tangential divergent} $div_S\stackrel{\rightarrow}{a}:S\rightarrow\mathbb{R}$
 as the unique continuous function in $S$ such that, for any   $\phi :S\rightarrow\mathbb{R}$ of  $\mathcal{C}^1$ with compact support in  $S$
$$
\int_{S}\phi \, div_{S}\stackrel{\rightarrow}{a}=-\int_{S}\stackrel{\rightarrow}{a}\cdot\nabla\phi.
$$

If   $u:S\rightarrow\mathbb{R}$ is  of class $\mathcal{C}^2$ then its
 \textit{tangential Laplacian} is defined by $\Delta_{S}u=div_{S}\nabla_Su$. %or, equivalent, as the continuous  function in $S$ such that, for any %$\phi:S\rightarrow\mathbb{R}$ 
%$$
%\int_{S}\phi\Delta_Su=-\int_{S}\nabla_S\phi\cdot\nabla_Su .
%$$

\begin{teo}\label{tods} .\\
\begin{enumerate}
\item If  $S$ is a  $\mathcal{C}^1$ hypersurface  in $\mathbb{R}^n$ and  $\phi: \mathbb{R}^n\rightarrow\mathbb{R}$ is $\mathcal{C}^1$ in a neighborhood of  $S$, then $\nabla_S\phi(x)$ is the  component  $\nabla\phi$ tangent to $S$ at the point  $x$, that is  
$$
\nabla_{S}\phi=\nabla\phi-N\frac{\partial\phi}{\partial N},
$$ 
where  $N$ is an unit normal field on $S$.

\item If  $S$ is a  $\mathcal{C}^2$ hypersurface, $\stackrel{\rightarrow}{a}:\mathbb{R}^n\rightarrow\mathbb{R}^n$ is  $\mathcal{C}^1$ in a neighborhood of $S$, $N:\mathbb{R}^n\rightarrow\mathbb{R}^n $ is a  $\mathcal{C}^1$ unit-vector field on a neighborhood of $S$, which is a normal field at points of 
 $S$ near  $x_0 \in S$, and  $H =div N$ is the mean curvature of
  $S$ (near $x_0$), then
$$
div_{S}\stackrel{\rightarrow}{a}= div\stackrel{\rightarrow}{a}-H(x)\stackrel{\rightarrow}{a}\cdot N-\frac{\partial}{\partial N}(\stackrel{\rightarrow}{a}\cdot N) 
$$
 $S$ (near $x_0$).

\item  If  $S$ is $\mathcal{C}^2$ hypersurface $u:\mathbb{R}^n\rightarrow\mathbb{R}$ is  $\mathcal{C}^2$ on a neighborhood of $S$, and $N$ is as in 2) above, then
$$
\Delta_{S}u=\Delta u -divN\frac{\partial u}{\partial N}-\frac{\partial^2u}{\partial N^2}+\nabla_Su\cdot\frac{\partial N}{\partial N}.
$$
on $S$ near $x_0$.  We may choose $N$ so that $\frac{\partial N}{\partial N}=0 $ and then the final term is omitted. $\Delta_{S}u$ depends only on the values of $u$ on $S$.
\end{enumerate}
\end{teo}
%\begin{proof} See \cite{hp}.
 %\end{proof}

\begin{teo}
Let $\Omega\subset\mathbb{R}^n$ be a $\mathcal{C}^2$-regular  domain  $h(t,.)$ a family of diffeomorphisms such that  $\frac{\partial }{\partial t}h(t,x)=V(t,h(t,x))$, $\frac{\partial^2h}{\partial x^2}, \frac{\partial^2h}{\partial t\partial x}$ are continuous and  $V\in \mathcal{C}^2(\mathbb{R}\times\mathbb{R}^n,\mathbb{R}^n)$. If $f:\mathbb{R}\times\mathbb{R}^n\rightarrow\mathbb{R}$ is $\mathcal{C}^1$ then, for small $t$ 
 $t\mapsto\int_{\partial\Omega(t)}f(t,x)dA_x$ is $\mathcal{C}^1$ and
$$
\frac{d}{dt}\int_{\partial\Omega(t)}f(t,x)dA_x=\int_{\partial\Omega(t)}\left(\frac{\partial f}{\partial t}+ V\cdot N\frac{\partial f}{\partial N}+HV\cdot Nf\right)dA_x,
$$
where $N$ is the unit outward  normal $\partial\Omega(t)$ and  $H=div N$.
\end{teo}

The following uniqueness result   will be frequently needed. 

%\begin{teo}\label{tuclbl}
%Let  $\Omega\subset\mathbb{R}^n$ be an open, connected, bounded  $\mathcal{C}^4$-regular region, and  $B$  an open ball in  $\mathbb{R}^n$ such that  $B\cap\partial\Omega$  $\mathcal{C}^4$ hypersurface.
 %Suppose that $u\in H^4(\Omega)$ satisfies 
%$$
%|\Delta^2u|\leq C(|\Delta u |+|\nabla u|+|u|)\,\, q.t.p \,\,{\mathrm{em }}\,\,\Omeg%a 
%$$
%for some positive constant  $C$ and 
%$$
%u=\frac{\partial u}{\partial N}=\Delta u=\frac{\partial(\Delta u)}{\partial N}=0 \,%\,{\mathrm{em}}\,\,B\cap\partial\Omega.
%$$
%Then  $u$  vanishes in  $\Omega$.
%\end{teo}

\begin{teo}\label{tuc}\textit{Uniqueness in the Cauchy Problem}
Let  $\Omega\subset\mathbb{R}^n$ be an open, connected, bounded  $\mathcal{C}^2$-regular region, and  $B$  an open ball in  $\mathbb{R}^n$ such that  $B\cap\partial\Omega$  $\mathcal{C}^2$ hypersurface.
 Suppose that $u\in H^2(\Omega)$ satisfies 
$$
|\Delta u|\leq C(|\nabla u|+|u|)\,\, a.e \,\,{\mathrm{in}}\,\,\Omega 
$$
for some positive constant  $C$ and 
$$
u=\frac{\partial u}{\partial N}=0 \,\,{\mathrm{on}}\,\,B\cap\partial\Omega.
$$
Then  $u$  vanishes in  $\Omega$.

\end{teo}

\subsection{Perturbation of domains } \label{secperturb}

Given an open, bounded, ${C}^m$ region  $\Omega_{0}  
 \subset \R^{n}$, consider the following open subset of  
 ${C}^m(\Omega, \R^n)$   
$$ 
 \mathrm{Diff}^m(\Omega) = \{ h \in {C}^m(\Omega, \R^n) \; | \; h \textrm{ 
   is injective  and  } \frac{1}{| det h'(x) |} \textrm{ is bounded in  } 
 \Omega  \}.  
$$  
and the collection of regions  
$\{ h(\Omega_0) \; | \; h \in  \mathrm{Diff}^m(\Omega_0) \}$. 
We introduce a topology in  this set  by defining a (sub-basis of) the 
neighborhoods of a given $\Omega$ by   
$$ 
\{ h(\Omega) ; \|h -i_{\Omega}\|_{{C}^m(\Omega,\R^n)} < \varepsilon, 
\varepsilon >0 \textrm{ sufficiently small}  \},  
$$  
where $i_{\Omega}:\Omega \mapsto \R^n$ is the inclusion.
When $ \|h -i_{\Omega}\|_{{C}^m(\Omega,\R^n)} $ is small, $h$ is 
a  $ {C}^m$ embedding of $\Omega $ in $\R^n$, a  
${C}^m$ diffeomorphism to its image $h(\Omega)$. 
Micheletti \cite{micheletti-metrica} shows this topology is metrizable, and the 
set of regions $ {C}^m$-diffeomorphic to $\Omega$ may be 
considered a complete and separable metric space which we denote by 
$\mathcal{M}_{m}(\Omega)=\mathcal{M}_{m}$.    
We say that a function $F$ defined in the space $\mathcal{M}_{m}$ 
with values in a Banach space is ${C}^m$ or analytic  
if  $ h \mapsto F(h(\Omega)) 
$ is ${C}^m$ or analytic as a map of Banach spaces  
($h$ near $i_{\Omega}$ in  ${C}^m(\Omega, \R^n)$).  In this sense, 
we may express problems of perturbation of the boundary of a boundary value 
problem as problems of differential calculus in Banach spaces.

Consider the formal linear differential operator 

$$
Lu(x)=\left(u(x),\frac{\partial u}{\partial x_1}(x),...,\frac{\partial u}{\partial x_n}(x),\frac{\partial^2 u}{\partial x^2_1}(x),\frac{\partial^2 u}{\partial x_2\partial x_1}(x),...\right),\,\,x\in\mathbb{R}^n,
$$
$Lu(x)\in\mathbb{R}^p$. Given a function  $f:O\subset\mathbb{R}^n\times\mathbb{R}^p$, where $O$ is open, writing
  
$$
v(x)=f(x,Lu(x)),
$$
 one can define, for any open set $\Omega \in \R^n$, the  nonlinear differential operator  
 $F_{\Omega}$ by 
$$
F_{\Omega}=f(x,Lu(x)), \,\,x\in\Omega,
$$
 for sufficiently smooth functions defined in $\Omega$, with $(x,Lu(x))\in O$, for any  $x\in$ $\stackrel{-}{\Omega}$. If  $f$ is continuous, $\Omega$ is bounded and the differential operator  $L$ is of order less or equal than  $m$, the domain of $F_{\Omega} $ is a non empty open subset of $\mathcal{C}^m(\Omega)$ with values in  $\mathcal{C}^0(\Omega)$, that is 
\begin{equation*}
\begin{array}{llr}
           &   & F_{\Omega}: D_{F_{\Omega}}\subset \mathcal{C}^m(\Omega)\rightarrow \mathcal{C}^0(\Omega)\\
           &   & u\mapsto f(x,Lu(x)).
\end{array}
\end{equation*}

Let  $h:\Omega\rightarrow\mathbb{R}^n$ be a $\mathcal{C}^m$ embedding. If  $u$ is defined in  $h(\Omega)$, we define the composition or "pull-back" map by 
%$$
%h^*u(x)=(u\circ h)(x)= u(h(x)),\,\,x\in\Omega.
%$$
 %The map 
 $$
\begin{array}{ccc}
     &   & h^*: \mathcal{C}^m(h(\Omega))\rightarrow \mathcal{C}^m(\Omega)\\
     &   &   u\mapsto u\circ h
\end{array}
$$ 
which is then an isomorphism with inverse 
$h^{*-1}=(h^{-1})^*$.
We use the same notation for the pull-back in other function spaces.
%\begin{prop}
%A aplicação
%$$
%\begin{array}{ccc}
 %    &   & h^*: \mathcal{C}^m(h(\Omega))\rightarrow %\mathcal{C}^m(\Omega)\\
%     &   &   u\mapsto u\circ h 
%\end{array}
%$$
%é um isomorfismo com inversa $h^{*-1}=(h^{-1})^*$.
%\end{prop}
%\begin{proof}
%Se $h$ é um difeomorfismo de classe $\mathcal{C}^m$ temos, pela regra da cadeia, que $h^*$ está bem definido e é linear, injetiva e sobrejetiva. De fato, é também limitada, pois $||h^*u||_{\mathcal{C}^m(\Omega)}=||u\circ h||_{\mathcal{C}^m(\Omega)}\leq c||u||_{\mathcal{C}^m(h(\Omega))}$ para algum $c>0$.
 
%\end{proof}
 
 If $h$ is such an embedding we can consider the differential operator 
 acting on the perturbed region $h(\Omega)$

$$
F_{h(\Omega)}: D_{F_{h(\Omega)}}\subset \mathcal{C}^m(h(\Omega))\rightarrow \mathcal{C}^0(h(\Omega)).
$$
which is termed \textit{the Eulerian form} of the formal nonlinear differential operator 
$v \mapsto f( \cdot ,Lv(\cdot)), \,\,x$   on $h(\Omega)$,
while

$$
h^*F_{h(\Omega)}h^{*-1}: h^*D_{F_{h(\Omega)}}\subset \mathcal{C}^m(\Omega)\rightarrow \mathcal{C}^0(\Omega)
$$
is called its \textit{Lagrangean} form.

We also treat boundary conditions in the same way. The Neumann problem  requires $N_{\Omega(t)}(y)\cdot\nabla u=0 $ on $\partial\Omega(t)$ in this case  the particular extension of $N_{\Omega(t)}$ away from the boundary is irrelevant. We choose some extension of $N_{\Omega}$ in the reference region and then define $N_{\Omega(t)}=N_{h(t,\Omega)}$ by 
\begin{equation}\label{Nf}
h^*N_{h(\Omega)}(x)=N_{h(\Omega)}(h(x))= ^Th_x^{-1}N_{\Omega}(x)\frac{1}{ ||^Th_x^{-1}N_{\Omega}(x)||},
\end{equation}

for $x\in\partial\Omega$, where $^Th_x^{-1}$ is the inverse-transpose of the 
Jacobian matrix $h_x=[\frac{\partial h_i}{\partial x_j}]_{i,j=1}^{n}$ and $||.||$ is the Euclidean norm. 
%$$
%\xymatrix{
%\mathcal{C}^m(h(\Omega))\ar[r]^{F_{h(\Omega)}}\ar[d]_{h^*} & \mathcal{C}^0(h(\Omega))\ar[d]^{h^*} \\
%\mathcal{C}^m(\Omega)\ar[r]^{h^*F_{h(\Omega)}h^{*-1}} & \mathcal{C}^0(\Omega)
%} 
%$$

The Eulerian form is more natural and, usually, more convenient for computations (see, for example,  Corollary  \ref{ed}) while the Lagrangean form is more appropriate to prove results (see section \ref{DCAA}).

 The advantage of the Lagrangean form is to act in spaces which don't depend on  $h$, which facilitates (for example) the use of the
  Implicit Function Theorem. However, we then need to know the smoothness of  
\begin{equation}\label{aod}
(u,h)\mapsto h^*F_{h(\Omega)}h^{*-1},
\end{equation}
 and we need to be able to compute derivatives with respect to 
 $h$.
% To this end, it is necessary to explicit the domain of the parameter $h$.
% If $\Omega\subset\mathbb{R}^n$ is a $ \mathcal{C}^m$-regular bounded open domain then 
%$$
%Diff^m(\Omega)=\{h\in \mathcal{C}
%^m(\Omega,\mathbb{R}^n)/ h{\mathrm{ \,\, \acute{e} \,\, injective \,\,and ,\,}} |deth_x(x)|^{-1}\,\, {\mathrm{\acute{e} \,\,limitado\,\, em \,\,}}\Omega\}.
%$$
 It is shown by Henry in \cite{hp} that the  map  (\ref{aod}),
 from  $ Diff^m(\Omega) \times \mathcal{C}^m(\Omega)$ into  $  \mathcal{C}^0$
  is as regular as the  function $f$ (other function spaces can also be used,
 with similar results).

The next result is used throughout the paper.

\begin{lema}\label{fdn}
Let $\Omega$ a $C^2$-regular region, $N_{\Omega}(.)$ a $C^1$ unit-vector field defined on a neighborhood of $\partial\Omega$ which is the outward normal on $\partial\Omega$, and for $C^2$ embeddings $h:\Omega\rightarrow\mathbb{R}^n$ define $N_{h(\Omega)}$ on a neighborhood of $h(\partial\Omega)=\partial h(\Omega)$ by (\ref{Nf}) above. Suppose $h(t,.)$ is an embedding for each $t$, defined by 
$$
\frac{\partial}{\partial t}h(t,x)= V(t,h(t,x))\,\, \,\,x\in\Omega,\, h(0,x)=x,
$$
$(t,x) \rightarrow V(t,x)$ is $C^2$ and $\Omega(t)=h(t,\Omega)$, $N_{\Omega(t)}=N_{h(t,\Omega)}$. Then for $x$ near $\partial\Omega$, $y=h(t,x)$ near $\partial\Omega(t)$,
\begin{eqnarray*}
\left(\frac{\partial }{\partial t}\right)N_{\Omega(t)}(y)= -(\nabla_{\partial\Omega(t)}\sigma + \sigma \frac{\partial N_{\Omega(t)} }{\partial N_{\Omega(t)}}),
\end{eqnarray*}
$\sigma=V\cdot N_{\Omega(t)}$ is the normal velocity and $\nabla_{\partial\Omega(t)}=\nabla - N_{\Omega(t)}\frac{\partial \sigma}{\partial N_{\Omega(t)}}$ is the component of  the gradient tangent to $\partial\Omega(t)$.
\end{lema}

\section{Continuity and analiticity of curves of eigenvalues} \label{DCAA}

 In this section,we present some results on the  continuity of the  eigenvalues of the Neumann Laplacian
   with respect to $\mathcal{C}^2$ perturbations of the domain and  in the case of parametrized families of 
    $\mathcal{C}^{2}$ domains   we prove the existence of analytic curves of eigenvalues and eigenfunctions. Although these results could probably be obtained adapting results in \cite{kato}, we found it easier to follow the   approach of Henry (see examples 4.1 and 4.4 of \cite{hp}) which relies on a careful use of the Lyapunov-Schimdt method.

 We also obtain expressions  for the first and second derivatives of the eigenvalues in this case.

\subsection{Continuity}
  
We consider here the slightly more general case of the
  Laplace problem with Robin  boundary conditions in a  regular  bounded open region    $\Omega\subset\mathbb{R}^n$. 
           \begin{equation}\label{plr}
                 \left\{
                       \begin{array}{lccc}
                             (L +\lambda)u=0,& in \ \ \Omega;\\
                              (\frac{\partial }{\partial N} +
 \beta(x)) u=0, & \ \ on \ \ \partial \Omega;
                       \end{array}
                            \right.
            \end{equation}                
            where
 $L= \Delta+c(x)$ and  $c$ and   $\beta$ are  of class  $\mathcal{C}^2$. 

 The associated  Lagrangean form is then
           \begin{equation}\label{plrp}
                 \left\{
                       \begin{array}{lccc}
                             h^*(L+ \lambda)h^{*-1}u=0,& in \ \ \Omega;\\
                              h^*(\frac{\partial }{\partial N_{h}}+\beta(x))h^{*-1} u=0, & \ \ on\ \ \partial \Omega;
                       \end{array}
                            \right.
            \end{equation}                
 where $h\in Diff^2(\Omega)$. The regularity of the perturbed problem with respect to  $h$ depends on the regularity of the functions
   $c$  and $\beta$. More precisely, if $\Omega$, $h\in Diff^m(\Omega)$, $c\in {\cal{C}}^{r+m-2}$ and $\beta\in {\cal{C}}^{r+m-1}$   , then for $u\in H^m(\Omega)$

$$
(h,u)\longmapsto h^*(\Delta+c)h^{*-1}u\in H^{m-2}(\Omega),
$$                            
 is  of class  $\mathcal{C}^r$ and 
 $$
\begin{array}{c}
(h,u)\longmapsto h^*(\frac{\partial}{\partial N_h}+\beta)h^{*-1}u\in H^{m-\frac{3}{2}}(\partial\Omega)
\end{array}
$$
 is  of class  $\mathcal{C}^r$
  since
 $(h,u)\longmapsto (c \circ h)u\in H^{m-r}(\Omega)$ is of class $\mathcal{C}^r$ and 
  $(h,u)\longmapsto (\beta\circ h)u\in H^{m-r-1}(\Omega)$ is of class $\mathcal{C}^r$.
(in the purely Neumann case, we obtain that both maps are of class
 $\mathcal{C}^1$ requiring $h$ of class $\mathcal{C}^2$) (see \cite{hp}, Example 3.2).

% \begin{teo}\label{tcaln}
% Suppose  $\lambda_0$  is the unique eigenvalue of f (\ref{plr}) in the interval $(\lambda_0-\epsilon,\lambda_0+\epsilon)$.
% If $\lambda_0$ has   multiplicity $m>1$  then there exists  $\delta>0$
%  such that, for all  $h\in Diff^2(\Omega)$, $||h-i_{\Omega}||_{\mathcal{C}^2}<\delta$, there exist exactly $m$ eigenvalues (counting multiplicity)   of the problem (\ref{plrp})  in  $(\lambda_0-\epsilon,\lambda_0+\epsilon)$. 
% \end{teo}
  
% \begin{proof}
%  Let $\tilde{L}$ be the extension of $L$ as an operator from $H^1$
%  to $H^{-1}$. Then, it proved in
%  \cite{contin} that  
%    $$  || \tilde{L}h^{*-1}u -   \tilde{L} ||_ u ||_{H^{-1}_{\Omega}}
%   \leq \varepsilon{h}
%   ||u||_{H^1{\Omega}} + \eta(h) ||\tilde{L}u||_{H^{-1}_{\Omega}},
%   $$
%  where $\varepsilon(h)$ and $\eta(h)$ go to $0$ as 
% $ ||h - i_{\Omega}||_{\mathcal{C}^2}$ goes to $0$.

%   Since the eigenvalues (and eigenfunctions) of $L$ and its extension 
% $\tilde{L}$ are the same, the result follows  then from  IV-2.14 and  IV-3.16
%  of \cite{kato}
 
%  \end{proof}

\begin{teo}\label{tcaln}
Suppose  $\lambda_0$  is the unique eigenvalue of  (\ref{plr}) in the interval $(\lambda_0-\epsilon,\lambda_0+\epsilon)$.
If $\lambda_0$ has   multiplicity $m$  then there exists  $\delta>0$
 such that, for all  $h\in Diff^2(\Omega)$, $||h-i_{\Omega}||_{\mathcal{C}^2}<\delta$, there exist exactly $m$ eigenvalues (counting multiplicity)   of the problem (\ref{plrp})  in  $(\lambda_0-\epsilon,\lambda_0+\epsilon)$. 
\end{teo}

\begin{proof}
 Let  $\{\phi_j\}_{j=1}^m$ be an orthonormal basis for the eigenspace associated to  $\lambda_0$  and  $Pu=\sum_j^m\phi_j\int_{\Omega}\phi_j u$  the orthogonal projection into it. We write an arbitrary function $u \in L^2 (\Omega)$ in a unique way as  $u=\phi+\psi$, where  $\phi\in{\cal{R}}(P)= {\cal N}( L +\lambda_0)$  and  $\psi\in {\cal{N}}(P)= {\cal R}(L + \lambda_0)$. The perturbed problem (\ref{plrp})
 is then equivalent to the equations   

           \begin{equation}\label{deqlr}
                \left\{
                      \begin{array}{lccc}
                            P(h^*(L+ \lambda)h^{*-1}(\phi+\psi))=0,& \textrm{in}
  \ \ \Omega;\\
                            (I-P)(h^*(L+ \lambda)h^{*-1}(\phi+\psi)=0,& 
 \textrm{in}  \ \ \Omega;\\
                             h^*(\frac{\partial }{\partial N}+\beta(x))h^{*-1} (\psi+\phi)=0, & \ \  \textrm{on} \ \ \partial \Omega;
                      \end{array}
                           \right.
           \end{equation}                

 We first solve the second and third equations. The boundary term can be
 rewritten as  %somando e subtraindo o termo $(\frac{\partial}{\partial N}+\beta)\psi$:

\begin{equation*}
\left(\frac{\partial }{\partial N}+\beta\right)\psi+\left(h^*\left(\frac{\partial }{\partial N}+\beta\right)h^{*-1}-\left(\frac{\partial}{\partial N}+\beta\right)\right)(\psi+\phi)=0.
\end{equation*}

 Now, summing and subtracting the therm  $(L+ \lambda)\psi$ in the second equation and observing that  

\begin{eqnarray}
PL\psi=P(L+\lambda)\psi & = & \sum_{j=1}^m\phi_j\int_{\Omega}\phi_j(L+\lambda)\psi\nonumber\\
& = & \sum_{j=1}^m\phi_j\left(\int_{\Omega}\phi_j(L+\lambda)\psi-\psi(L+\lambda)\phi_j\right)\nonumber\\
& = & \sum_{j=1}^m\phi_j\left(\int_{\partial\Omega}\phi_j\frac{\partial\psi}{\partial N}-\psi\frac{\partial\phi_j}{\partial N}\right)\nonumber\\
& = & \sum_{j=1}^m\phi_j\left(\int_{\partial\Omega}\phi_j\left(\frac{\partial}{\partial N}+\beta\right)\psi-\psi\left(\frac{\partial}{\partial N}+\beta\right)\phi_j\right)\nonumber\\
& = & \sum_{j=1}^m\phi_j\int_{\partial\Omega}\phi_j\left(\frac{\partial}{\partial N}+\beta\right)\psi,\nonumber
\end{eqnarray}
 and 
\begin{eqnarray}
(L+\lambda)\psi & = & (I-P) \left[(L+\lambda)\psi\right] +\sum_{j=1}^m\phi_j\int_{\Omega}\phi_j(L+\lambda)\psi\nonumber,
\end{eqnarray}
we obtain 
$$
(L+ \lambda)\psi+ (I-P)(h^*Lh^{*-1}-L)(\psi+\phi)-\sum_{j=1}^m\phi_j\int_{\partial\Omega}\phi_j\left(\frac{\partial}{\partial N}+\beta\right)\psi=0.
$$

Therefore, the second and third equations are equivalent to 
 $F(h,\lambda, \phi,\psi)=0$, where 

\begin{eqnarray}
F\ : & \! Diff^2(\Omega)\times\mathbb{R}\times{\cal{R}}(P)\times H^2(\Omega)\cap{\cal{N}}(P)\longrightarrow &{\cal N}(P)\times H^{\frac{3}{2}}(\Omega)\nonumber\\
     & \!F(h,\lambda,\phi,\psi)\!=\!(F_1(h,\lambda,\phi,\psi),F_2(h,\lambda,\phi,\psi))\nonumber
\end{eqnarray}     
           
and          
           \begin{equation*}
                \left\{
                      \begin{array}{lcc}
                          F_1=(L+ \lambda)\psi+ (I-P)(h^*Lh^{*-1}-L)(\psi+\phi)-\sum_{j=1}^m\phi_j\int_{\partial\Omega}\phi_j(\frac{\partial}{\partial N}+\beta)\psi,\\
                          F_2=(\frac{\partial }{\partial N}+\beta(x))\psi+(h^*(\frac{\partial }{\partial N}+\beta)h^{*-1}-(\frac{\partial}{\partial                                  N}+\beta)) (\psi+\phi).
                      \end{array}
                           \right.
           \end{equation*}

 Now, since the map 

$$
\begin{array}{c}
\frac{\partial F}{\partial \psi}(i_{\Omega},\lambda_0,0,0)\dot{\psi}=((L+ \lambda_0)\dot{\psi}-\sum_{j=1}^m\phi_j\int_{\partial\Omega}\phi_j(\frac{\partial}{\partial N}+\beta)\dot{\psi},(\frac{\partial }{\partial N}+\beta(x))\dot{\psi})
\end{array}
$$
is an isomorphism  from 
 $H^2(\Omega)\cap{\cal{N}}(P)$  into  ${\cal N}(P)\times H^{\frac{1}{2}}(\Omega)$.
It follows from the Implicit Function Theorem, that 
 the equation $F(h,\lambda,\phi,\psi)=(0,0)$ can be solved for  $\psi$ as a function of  $\lambda$, $h$ and $\phi$.
More precisely, there exist neighborhoods 
 $\cal{V}$ in  $\mathcal{C}^2(\mathbb{R}^n,\mathbb{R}^n)$  of  $i_{\Omega}$, $(\lambda_0-\epsilon, \lambda_0+\epsilon)$ of  $\lambda_0$ and a
  $\mathcal{C}^1$ function  $ \psi= S(h,\lambda)\phi$ which gives the unique solution of     
$F(h,\lambda,\phi,\psi)=0$, with   $h \in  \cal{V}$ and $\lambda \in
 (\lambda_0-\epsilon, \lambda_0+\epsilon)$ . Furthermore, $S(h,\lambda)\phi$ is
 analytic  $\lambda$  and  linear  in  $\phi$.

 Now, to solve the first equation in  (\ref{deqlr}) observe that, since
 $\phi\in{\cal{R}}(P)$, there exist  real numbers $c_1,c_2,...,c_m$ not all equal to  zero, such that  $\phi = \sum_{j=1}^{m}c_j\phi_j$ and, therefore, the equation  (\ref{deqlr})  is equivalent to the system in the variables $c_1,...,c_j$
$$
\sum_{j=1}^m c_j\int_{\Omega}\phi_kh^*(L+\lambda)h^{*-1}(\phi_j+S(h,\lambda)\phi_j)=0
$$
for $k=1,2,...,m$.  Thus,  $\lambda$ is an eigenvalue of  (\ref{plrp})
 if, and only if  $Det M(h,\lambda)=0$, where 
$$
M_{k,j}(h,\lambda)=\int_{\Omega}\phi_kh^*(L+\lambda)h^{*-1}(\phi_j+S(h,\lambda)\phi_j).
$$ 
and, in this case, the associated eigenfunctions are given by 

$$
u=\sum_{j=1}^{m}c_j(\phi_j+S(h,\lambda)\phi_j),
$$
where $c=(c_1,...,c_m)$ satisfies $M(h,\lambda)c=0$.

  Finally, we observe that the equation $DetM(h,\lambda)=0$ has exactly $m$
 roots in a neighborhood   ${\cal{V}}\times B_{\delta}(\lambda_0)$ of $(i_{\Omega},\lambda_0)$, by Rouche's theorem  since 
 $h=i_{\Omega}$, $Det M(i_{\Omega},\lambda)=(\lambda-\lambda_0)^m$ if
  $h=i_{\Omega}$. 

 \end{proof}

\subsection{Existence of analytic curves}

  The next result ensures the existence of analytic curves of eigenvalues and eigenvectors for the problem  $(\ref{plrp})_{h(t,.)}$  when  $h(t,.)$ 
 is an analytic curve of diffeomorphisms  if $c\equiv 0$ and $\beta\equiv 0$, that is for 
 the Neumann Laplacian.  

\begin{teo}\label{tecapln}
Suppose  $\lambda_0$  is an eigenvalue of multiplicity $m$ for 
 the problem (\ref{plr}) with 
 $c \equiv 0$, $\beta\equiv 0$, and let  $h(t,.)$ be an analytic curve of diffeomorphisms of
 class  $\mathcal{C}^{3}$ such that  $h(0,x)=x$.
 Then, there exist   $m$ analytic curves   $\mu_1(t),
 \mu_2(t), \cdot, \mu_m(t) $  and $m$ analytic curves  
 $\phi_1(t), \phi_2(t), \cdot, \phi_m(t) $, giving the eigenvalues and eigenfunctions of   $(\ref{plrp})_{h(t,.)}$  near $\lambda_0$ and its associated eigenfunctions.     
\end{teo}

%\begin{obs}
%Para provar a existência das curvas analíticas de autovalores bastaria considerar a equação $Det M(h(t,.),\lambda)=0$ obtida na demonstração do teorema anterior e utilizar o Teorema de Puiseux para funções analíticas (ver apêndice). Assim, dada uma curva de autovalores $\lambda(t)$, vemos que $0$ é autovalor para a matriz $M(t,\lambda(t))$ para cada $t$, portanto basta obter uma curva analítica de autovetores $c(t)$ associada ao autovalor $0$. Porém como a matriz $M$ não é necessariamente simétrica não temos  garantia da existência da curva analítica $c(t)$. Faremos uma construção ligeiramente diferente para a matriz $M$ dada na demonstração do teorema anterior de forma que a nova matriz  seja simétrica, desta forma conseguimos garantir a existência da curva $c(t)$ de autovetores, associados ao autovalor $0$ da matriz $M$. Para este fim utilizaremos os resultados encontrados em \cite{kato} para perturbação de operadores em espaços de dimensão finita.
%Para ver isto basta considerar o exemplo:
%$$
%A(x)=\left[\begin{array}{ccc}
%x & 1\\
%0 & 0
%\end{array}
%\right],
%$$ 
%neste caso $0$ é autovalor para todo $x\in\mathbb{C}$ e a autoprojeção associada é 
%$$
%P(x)=\left[\begin{array}{ccc}
%0 & -x^{-1}\\
%0 & 1
%\end{array}
%\right]
%$$
%para $x\neq0$.
%\end{obs}
 \begin{proof}
% Para provar a existência das curvas analíticas de autovalores bastaria considerar a equação $Det M(h(t,.),\lambda)=0$ obtida na demonstração do teorema anterior e utilizar o Teorema de Puiseux para funções analíticas (ver \cite{wall}). Porém a existência de uma curva analítica de autofunções associada não segue tão diretamente. A razão para isto é a seguinte: o teorema anterior mostra que as autofunções associadas a uma dada curva de autovalores $\lambda(t)$ são dadas por $u=\sum_{j=1}^{m}c_j(\phi_j+S(h(t,.),\lambda(t))\phi_j)$, onde $c=(c_1,...,c_m)\in\mathbb{R}^m$ é um autovetor da matriz $M(h(t,.),\lambda(t))$ associado ao autovalor $0$ e claramente $c$ depende de $t$. Portanto, já que a função $S(h(t,.),\lambda(t))\phi_j$ é analítica em $t$, a curva de autofunções $u(t)$ é analítica se, e somente se $c(t)$ é analítica também. No entanto, como a matriz $M$ não é simétrica não podemos obter o resultado diretamente da Teoria de Perturbação Analítica de Operadores Lineares em Dimensão Finita. Para contornarmos o problema faremos uma construção, seguindo a mesma linha do teorema anterior, de forma que a matriz $M$ seja simétrica. 

 Let  $\{\phi_j\}_{j=1}^m$ be an orthonormal basis of eigenfunctions of (\ref{plr}) associated to  $\lambda_0$.
 For each  $j=1,..m$, consider the problem 
           \begin{equation}\label{pfp}
                 \left\{
                       \begin{array}{lccc}
                             (L +\lambda_0)u=0,& in \ \ \Omega;\\
                              h^*\frac{\partial }{\partial N_h}h^{*-1}(\phi_j+ u)=0, & \ \ on\ \ \partial \Omega;\\
                              Pu=\sum_{j=1}^m\phi_j\int_{\Omega}\phi_ju=0.
                       \end{array}
                            \right.
            \end{equation}

 Consider the map
$$
F^j: Diff^3(\Omega)\times H^2(\Omega)\longrightarrow 
 [\phi_1, \phi_2, \cdots, \phi_m]^{\bot}\times{\cal{R}}(P)\times H^{\frac{1}{2}}(\partial\Omega)
:$$
$$
F^j(h,\omega)=((L+\lambda_0)\omega, P\omega, h^*\frac{\partial }{\partial N_h}h^{*-1} (\phi_j+\omega)),
$$                             

where $[\phi_1, \phi_2, \cdots, \phi_m]^{\bot} $  is the orthogonal complement to  ${\cal N}(L+\lambda_0)$ (with homogeneous Neumann boundary condition) in $L^2(\Omega)$. Since  $\frac{\partial F^j}{\partial \omega}(i_{\Omega},0)$
 is an isomorphism,   the Implicit Function Theorem ensures the existence
 of a neighborhood  $\cal{V}$  of  $i_{\Omega}$ in ${\cal{C}}^3(\R^n,\R^n)$ and an analytic function $\omega_j(h)$ on  $\cal{V}$  such that  $\omega_j(h)$ is the unique solution of  $F^j(h,\omega)=0  $, 
 for $h \in \cal{V}$.
 
  In this way we obtain, for each  $h$ in  $\cal{V}$,  a set
 $\{\varphi_j(h)\}_{j=1}^m$, $\varphi_j(h)=\phi_j+\omega_j(h)$, 
of linearly independent solutions of  $(\ref{pfp})$. Using the 
 Gram-Schmidt method, we can produce a new set of solutions  
  $\{\hat{\varphi_j}(h)\}_{j=1}^m$  which is orthonormal with respect to 
  the inner product  $(u,v)_h=\int_{\Omega}uv \, deth_x\, dx$. We observe that the  $\hat{\varphi_j}(h)$ belong to the domain of the operator $h^*L h^{*-1}$, \quad
  $D_h=\{u\in H^{2}(\Omega), h^*\frac{\partial}{\partial N_{h}}h^{*-1}u=0\}$. Furthermore, since with this inner product this operator is self-adjoint, it follows that the matrix given by  $\int_{\Omega}\hat{\varphi}_jh^*L h^{*-1}\hat{\varphi}_k deth_xdx$ is symmetric.

Consider now 
 the family of diffeomorphisms  $h(t,x)=x+tV(x)$  for some 
 $V\in \mathcal{C}^{3}(\mathbb{R}^n,\mathbb{R}^n)$ and  the family of projections

$$P(t)u=\sum_{j=1}^m\hat{\varphi}_j(t)\int_{\Omega}u\hat{\varphi_j}(t)deth_x(t,.)\, dx.$$        
  
 Define the map
  \begin{eqnarray*}
 G_j= (G_{1,j}, G_{2,j}, G_{3,j})\ : & \!(-\epsilon, \epsilon)\times\mathbb{R}\times H^2(\Omega)\longrightarrow &L^2(\Omega)\times H^{\frac{3}{2}}(\Omega)\times L^2(\Omega)\nonumber\\
 \end{eqnarray*}     
where            
            \begin{equation} \label{eqnG}
                 \left\{
                       \begin{array}{lcc}
                           G_{1,j}=(I-P(t))(h^*(t,.)(L+\lambda)h^{*-1}(t,.))(\omega+\hat{\varphi}_j(t))\\
                           G_{2,j}=h^*\frac{\partial }{\partial N_h}h^{*-1}\omega;\\
                           G_{3,j}= P(t)\omega,
                       \end{array}
                            \right.
            \end{equation}    
 Again by the Implicit Function Theorem, there exists a neighborhood                  $\cal{U}$  of $(0, \lambda_0)$  and an application $\omega_j(t,\lambda)$  which gives the unique solution of  $G_j(t,\lambda,\omega)=(0,0,0)$ in  $\cal{U}$. 
Since, for small $t$ and $\lambda$ near $\lambda_0$, the operator 
 $(I-P(t))(h^*(t,.)(L+\lambda)h^{*-1}(t,.))(\omega+\hat{\varphi}_j(t))$
 with  $ h^*\frac{\partial }{\partial N_h}h^{*-1}\omega =0 $ has an $m$ dimensional kernel, the solutions of the first and second equations will be of the form $ \sum_{j=1}^m c_j ( \hat{\varphi}_j(t)+\omega_j(t,\lambda))$.
 Therefore, a  number  $\lambda$  will be an eigenvalue  of  $(\ref{plrp})_{h(t,.)}$
  with eigenfunction
 $   \sum_{j=1}^m c_j (\hat{\varphi}_j(t)+\omega_j(t,\lambda))   $ if, and only if   $c=(c_1,...,c_m)$ is a nonzero vector such that
  $M(t,\lambda)c=0$, where 
\begin{equation*}
M_{ij}(t,\lambda)= \int_{\Omega}\hat{\varphi}_i(t)h^*(t,.)(L+\lambda)h^{*-1}(t,.)(\hat{\varphi}_j(t)+\omega_j(t,\lambda))deth_x(t,.).
\end{equation*}
that is , $\lambda$ is an eigenvalue if and only if  $DetM(t,\lambda)=0$.
 Now  
\begin{eqnarray*} M(t,\lambda) & = &\int_{\Omega} \left( 
 \hat{\varphi}_i(t) +\omega_i(t,\lambda) \right) 
  h^*(t,.) (L+\lambda)h^{*-1}(t,.)(\hat{\varphi}_j(t)+\omega_j(t,\lambda))deth_x(t,.) \\ 
 & - & \int_{\Omega} \omega_i(t,\lambda)(L+\lambda)h^{*-1}(t,.)(\hat{\varphi}_j(t)+\omega_j(t,\lambda))deth_x(t,.)   .
\end{eqnarray*}
 and the last term is zero by the first and third equations in  (\ref{eqnG}).
 It follows that M is symmetric  and Puiseux theorem \cite{wall} then ensures the existence of $m$ analytic curves
 $\lambda_1(t), \lambda_2(t), \cdots, \lambda_m(t)$  giving the $m$ (not necessarily distinct)  solutions of
   $DetM(t,\lambda)=0$. Since   $M$ is symmetric for each curve
  $\lambda_l(t)$, there also exists an  analytic curve
  $\mathcal{C}^l(t)\in\mathbb{R}^m$   of solutions of  $M(t,\lambda_l)\mathcal{C}(t)=0$, with
 $\mathcal{C}^1(t), \mathcal{C}^2(t), \cdots, \mathcal{C}^m(t)$ linearly independent. 
    Therefore, $$\psi^l(t)=\sum_{j=1}^m\mathcal{C}^l_j(t)(\hat{\varphi}_j(t)+\omega_j(t,\lambda_l(t))), \quad  l=1, \cdots, m$$
  
 is an analytic curve of associated eigenfunctions.

\end{proof}

\begin{obs} 
The above proof is similar to the argument in \cite{hp} example 4.4. However, here we needed to first construct solutions for the auxiliary problem (\ref{pfp}) since, otherwise, we would have not obtained a symmetric matrix  $M$. This is due to the fact that now the domain of the  operator $h^*L h^{*-1}$ varies with
 $h$.
  \end{obs}

Once we know the eigenvalues are analytic in the parameter $t$, its first and second derivatives can be obtained  using the methods developed in \cite{hp}.

\begin{coro}\label{ed}
Let  $\lambda_0$ be an eigenvalue of multiplicity $m$  of  (\ref{plrp}), with $\beta=c=0$  and  $h(t,.)+ I + t V(\cdot)$, with 
 $V$ of class $\mathcal{C}^2$  a curve of diffeomorphisms. Then,
 if $\lambda(t)$ is one of the  curves of eigenvalues given by
 theorem \ref{tecapln},  the derivatives
     $\dot{\lambda}= \frac{d}{dt}\lambda|_{t=0}$, $
 \ddot{\lambda}=\frac{d^2}{dt^2}\lambda|_{t=0}$
 satisfy  
$$
\begin{array}{l}
(\dot{\lambda}I+\stackrel{\circ}{M})c=0\\
(\ddot{\lambda} I+\stackrel{\circ\circ}{M})c+2(\dot{\lambda}I+\stackrel{\circ}{M})\dot{c}\;=0,
\end{array}
$$
for some $c$ and $\dot{c}$ in   $\mathbb{R}^n$. 
The matrices  $\stackrel{\circ}{M}$, $\stackrel{\circ\circ}{M}$
 are given by  
\begin{eqnarray}
\stackrel{\circ}{M}_{k,j} & = &\int_{\partial\Omega}\sigma(\nabla_{\partial\Omega}\phi_k\cdot\nabla_{\partial\Omega}\phi_j-\lambda_0\phi_k\phi_j)\nonumber\\
\stackrel{\circ\circ}{M}_{k,j}& = &\int_{\partial\Omega}2\sigma \dot{Q}_{jk}+\sigma^2\frac{\partial}{\partial N}Q_{jk}+\left[\frac{\partial\sigma}{\partial t}+\sigma\frac{\partial\sigma}{\partial N} + H\sigma^2\right]Q_{jk},\nonumber,
\end{eqnarray}
$
\begin{array}{l}
Q_{jk}=\nabla_{\partial\Omega}\phi_j\cdot\nabla_{\partial\Omega}\phi_k-\lambda_0\phi_j\phi_k\\
\dot{Q}_{jk}=\nabla_{\partial\Omega}\phi_k\cdot\nabla_{\partial\Omega}\dot{\phi_j}-\dot{\lambda}\phi_k\phi_j-\lambda\phi_k\dot{\phi_j},
\end{array}
$\vspace{0.3cm}\\
 where $\{\phi_j\}_{j=1}^m$  is an orthonormal basis for the eigenspace
 associated to  $\lambda_0$  and  $\dot{\phi_j}$ satisfies  \\ $\dot{\phi}_j\bot span[\phi_i]_1^{m}$,
\begin{equation*}
\left\{
\begin{array}{lc}
(\Delta+\lambda_0)\dot{\phi}_j\ \in span[\phi_i]_1^{m}; \\
\frac{\partial\dot{\phi}_j}{\partial N}=(div_{\partial\Omega}(\sigma\nabla_{\partial\Omega}\phi_j)+\lambda_0\sigma\phi_j),&  \ \ \textrm{on} \ \ \partial \Omega . 
\end{array}
\right.
\end{equation*}
\end{coro}

\begin{proof}

  We know that each eigenpair $(\lambda(t), v(t))$ satisfies
 
$$
\left\{
\begin{array}{cc}
(\Delta + \lambda(t))v(t,.)=0,& \textrm{in} \ \ \Omega_t;\\
\dfrac{\partial v(t,.)}{\partial N_{\Omega_t}}=0, & \ \ \textrm{on} \ \ \partial \Omega_t;
\end{array}
\right.
$$ 

Differentiating the first equation with respect to  $t$, at  $y=h(t,x)\in\Omega_t$, we obtain

$$
(\Delta + \lambda(t))(\frac{\partial}{\partial t}v(t,y)) +(\frac{d}{d t})\lambda(t)v(t,y)= 0
$$
From now on, we use the notation  $\stackrel{\cdot}{v}$ for the derivative
 $\frac{\partial}{\partial t} v(t,.)$ and also for any derivative with respect to   $t$.

 In the  boundary we have, for each $x\in \partial\Omega$

$$
\dfrac{\partial v(t,h(t,x))}{\partial N_{\Omega_t}}= N_{\Omega_t}(t,h(t,x))\cdot\nabla_y v(t,h(t,x)) =0
$$
where $\nabla_y$ is the derivative in the variable  $y=h(t,x)$.
 Differentiating with respect to 
  $t$, we obtain 

\begin{eqnarray}
0 & = & \frac{d}{dt}\left[\dfrac{\partial v(t,h(t,x))}{\partial N_{\Omega_t}}\right] \nonumber \\
 & = & \frac{d}{dt}[N_{\Omega_t}(t,h(t,x))]\cdot\nabla_y v(t,h(t,x))\nonumber \\ 
                                                                             &   &+N_{\Omega_t}(t,h(t,x))\cdot\frac{d}{dt}[\nabla_y                                                                                                       v(t,h(t,x))]\nonumber \\
                                                                             & = &\dot{N}_{\Omega_t}(t,y)\cdot\nabla_y                                                                                                       v(t,y)+N_{\Omega_t}(t,y)\cdot\nabla_y                                                                                                                  \dot{v}(t,y)+\nonumber \\                                                                                                            &   &\underbrace{                                                                                                                                           N_{\Omega_t}(t,y)\cdot\nabla_y^2v(t,y)V                                                                                         +\left(\nabla_y N_{\Omega_t}(t,y)V                                                                                        \right)\cdot\nabla_y                                                                                                                         v(t,y)}_{V(t,y)\cdot\nabla_y(N_{\Omega_t}\cdot\nabla_y v)} \nonumber                    
\end{eqnarray}

 From lemma \ref{fdn},

$$
\dot{N}_{\Omega_t}=-\nabla_{\partial\Omega t}\sigma - \sigma\dfrac{\partial N_{\Omega_t}}{\partial N_{\Omega_t}}
$$
where $\sigma = V(t,y)\cdot N_{\Omega_t}$. Since  $\dfrac{\partial v}{\partial N_{\Omega_t}}=0$ on $\partial\Omega_t$, it follows that  
$$
V\cdot\nabla_y \left(\dfrac{\partial v}{\partial N_{\Omega_t}}\right)=\sigma\frac{\partial}{\partial N_{\Omega_t}}\left(\dfrac{\partial v}{\partial N_{\Omega_t}}\right),\ \ 
\textrm{on} \ \ \partial\Omega_t.
$$
Thus
\begin{eqnarray}
\dfrac{\partial \dot{v}}{\partial N_{\Omega_t}}-\nabla_{\partial\Omega_t}\sigma\cdot\nabla_{\partial\Omega_t} v+\sigma\frac{\partial}{\partial N_{\Omega_t}}\left(\dfrac{\partial v}{\partial N_{\Omega_t}}\right)& = & 0\ \  \textrm{on}  \ \ \partial\Omega_t.\nonumber
\end{eqnarray}

%$
%\dfrac{\partial\stackrel{\cdot}{v}}{\partial N}= \nabla_{\partial\Omega}\sigma\cdot\nabla v - \sigma\dfrac{\partial^2 v}{\partial N_{\Omega}^2}
%$.
Now, using Theorem  \ref{tods}, we obtain

$$ 
\dfrac{\partial\dot{v}}{\partial N_{\Omega_t}}=div_{\partial\Omega_t}(\sigma\nabla_{\partial\Omega_t}v)+ \lambda(t)\sigma v,\ \ \textrm{on} \ \ \partial\Omega_t.
$$

Therefore $\dot{v}$ must satisfy the problem
\begin{equation}
\left\{
\begin{array}{lc}
(\Delta + \lambda(t))\dot{v} +\dot{\lambda}(t)v=0,& \ \  \textrm{in} \ \ \Omega_t;\\
\frac{\partial \dot{v}}{\partial N_{\Omega_t}}=div_{\partial\Omega_t}(\sigma\nabla_{\partial\Omega_t}v)+\lambda(t)\sigma v,&  \ \   \textrm{on} \ \ \partial \Omega_t;
\end{array}
\right.
\label{1deq}
\end{equation}

 We know that $v(0,.)=\sum_{j=1}^m c_j\phi_j$  for some scalars $c_j$, not all zero. Multiplying the equation  (\ref{1deq}) with $t=0$ by $\phi_k$  and integrating, we obtain
\begin{eqnarray*}
\dot{\lambda}c_k & = & -\int_{\Omega}\phi_k(\Delta+\lambda_0)\dot{v} \\ \nonumber
                             & = & \int_{\Omega}\dot{v}(\Delta+\lambda_0)\phi_k - \phi_k(\Delta+\lambda_0)\dot{v}   \\ \nonumber
                             & = & -\int_{\partial_\Omega}\phi_k(div_{\partial\Omega}(\sigma\nabla_{\partial\Omega}v)+\lambda_0\sigma v)) \\ \nonumber
                             & = & \int_{\partial\Omega}\sigma(\nabla_{\partial\Omega}\phi_k\cdot\nabla_{\partial\Omega}v-\lambda_0\phi_kv) \\\nonumber
                                                                                                                                                                 & = &\sum_{j=1}^m c_{j}\int_{\partial\Omega}\sigma(\nabla_{\partial\Omega}\phi_k
                                                                                                                                                            \cdot\nabla_{\partial\Omega}\phi_j-\lambda_0\phi_k\phi_j).
\end{eqnarray*}
Writing  $c=(c_1,c_2,...,c_m)$ and

$$
\stackrel{\circ}{M}_{k,j}=\int_{\partial\Omega}\sigma(\nabla_{\partial\Omega}\phi_k\cdot\nabla_{\partial\Omega}\phi_j-\lambda_0\phi_k\phi_j)
$$
 we see that  $(\stackrel{\circ}{M}-\dot{\lambda})c=0$ and, therefore, the derivative $\dot{\lambda}(t)$ is an eigenvalue  of the matrix   $\stackrel{\circ}{M}$.

 Now, to compute  $\ddot{\lambda}$, we need to differentiate  (\ref{1deq})
 once again. We start with the boundary condition
\begin{equation}
\frac{\partial \dot{v}}{\partial N_{\Omega_t}}-(div_{\partial\Omega_t}(\sigma\nabla_{\partial\Omega_t}v)+\lambda(t)\sigma v)=0
\label{fronteira}
\end{equation}
If $f(t,h(t,x))=0$, $\forall x\in \front{}$, with  $f$, we obtain, differentiating with respect to $t$ 
$$
\dot{f}(0,x)+\sigma \derivada{}{f}{N}(0,x)=0,\ \ on \,\front{}.
$$
 Applying this formula in the equation  (\ref{fronteira}), it follows that
\begin{eqnarray}
\dot{f}(0,x) & =               &\derivada{}{\ddot{v}}{N}-\nablafront{}\sigma\cdot\nablafront{}\dot{v}-\left[\derivada{}{}{t}\divfront{t}(\sigma\nablafront{t}v)+\derivada{}{}{t}(\sigma\lambda v)\right]\nonumber \\  
\sigma\derivada{}{f}{N}(0,x) & = & \sigma\derivada{2}{\dot{v}}{N}-\sigma\left[ \derivada{}{}{N}\divfront{}(\sigma\nablafront{}v)+\derivada{}{}{N}(\sigma\lambda v)\right].\nonumber
\end{eqnarray}
Thus

\begin{eqnarray}
\derivada{}{\ddot{v}}{N}& = & \nablafront{}\sigma\cdot\nablafront{}\dot{v}-\sigma\derivada{2}{\dot{v}}{N}+\sigma(\dot{\lambda}v+\lambda\dot{v})+\left[ \derivada{}{\sigma}{t}+\sigma\derivada{}{\sigma}{N}\right]\lambda v+ \nonumber \\
                       &   & +\derivada{}{}{t}\divfront{t}(\sigma\nablafront{t}v)+\sigma\derivada{}{}{N}\divfront{}(\sigma\nablafront{}v)\nonumber.    
\end{eqnarray}
Multiplying the equation (\ref{fronteira}) by  $-\sigma H$
 and summing with the above equation, we obtain the  boundary condition

\begin{eqnarray*}
\derivada{}{\dpt{v}}{N}& = & \divfront{}(\sigma\nablafront{}{\dot{v}})+2\sigma(\dot{\lambda}v+\dot{v}\lambda)+\left[\derivada{}{\sigma}{t}+\sigma\derivada{}{\sigma}{N}+\sigma^2H\right]\lambda v +\nonumber \\
                       &   &\,\derivada{}{}{t}\divfront{t}(\sigma\nablafront{t}v)+ \sigma\derivada{}{}{N}\divfront{}(\sigma\nablafront{}{v})+\sigma H\divfront{}(\sigma\nablafront{}v)
\end{eqnarray*}

 Now, differentiating the equation in the interior, we obtain 
\begin{eqnarray}
(\Delta+\lambda)\ddot{v}+2\dot{\lambda}\dot{v}+\ddot{\lambda}v=0.
\label{interior}
\end{eqnarray}

Thus, to compute the second derivative, we need know  $\umpt{v}$.
To this aim, we first observe that there is a unique  $\dot{\phi}_j\in H^2(\Omega)$, such that  $\dot{\phi}_j\bot[\phi_1, \phi_2, \cdots, \phi_m]$
\begin{equation*}
\left\{
\begin{array}{lc}
(\Delta+\lambda)\dot{\phi}_j\ \in span[\phi_i]_1^{m}\\
\derivada{}{\dot{\phi}_j}{N}=(div_{\partial\Omega}(\sigma\nabla_{\partial\Omega}\phi_j)+\lambda_0\sigma\phi_j),&  \ \ on\ \ \partial \Omega.
\end{array}
\right.
\end{equation*}
 Thus  $\dot{v}\!|_{t=0}=\sum_{j=i}^m c_j\dot{\phi}_j+\dot{c}_j\phi_j$, where the $\dot{c}_j$ are not all zero and the  $c_j$ as before.
Multiplying the equation (\ref{interior}) by  $\phi_k$ and integrating in $\Omega$, we have

\begin{eqnarray*}
 \ddot{\lambda}c_k+2\dot{\lambda}\dot{c}_k &= & -\int_{\front{}} \phi_k\derivada{}{\ddot{v}}{N}  \nonumber \\
 & = & -  \int_{\front{}}\phi_k\left(\divfront{}(\sigma\nablafront{}\dot{v})+2\sigma(\dot{\lambda} v+\lambda\dot{v})+\left[ \derivada{}{\sigma}{t}+\sigma\derivada{}{\sigma}{N}+\sigma^2H\right]\lambda v \right) \nonumber \\ 
                                              &   & - \int_{\front{}}\phi_k \derivada{}{}{t}\divfront{}(\sigma\nablafront{} v) + \phi_k\sigma\derivada{}{}{N}\divfront{}(\sigma\nablafront{} v)+\nonumber\\
                                              &   & +\sigma H\phi_k\divfront{}(\sigma\nablafront{}v).
\end{eqnarray*}

 It is convenient to write this expression in a different form. We split the computation in two parts.  We call  $I$ and $II$  the first and second integrals and start with the second.

Extending $\phi_k$ arbitrarily in a neighborhood of $\Omega$, we observe  that
\begin{eqnarray}
\dfrac{d}{dt}\left[\int_{\front{t}}\phi_k\divfront{t}(\sigma\nablafront{t}v)\right]\Big|_{t=0}& = &\int_{\front{}}\phi_k\derivada{}{}{t}\divfront{t}(\sigma\nablafront{t}v)\big|_{t=0}\nonumber\\ &  &+\sigma\derivada{}{}{N}(\phi_k\divfront{}(\sigma\nablafront{}v))\nonumber\\
&  &+\sigma H\phi_k\divfront{}(\sigma\nablafront{}v)=II.\nonumber 
\end{eqnarray}
On the other hand
 
$$
\int_{\front{t}}\phi_k\divfront{t}(\sigma\nablafront{t}v)=-\int_{\front{t}}\sigma\nablafront{t}\phi_k\cdot
\nablafront{t}v.
$$
 Thus
\begin{eqnarray}
II& = &-\dfrac{d}{dt}\left(\int_{\front{t}}\sigma\nablafront{t}\phi_k\cdot\nablafront{t}v\right)\Big|_{t=0}=-\int_{\front{}}\derivada{}{}{t}(\sigma\nablafront{t}\phi_k\cdot\nablafront{t}v)\Big|_{t=0}\nonumber \\
 &   & -\int_{\front{}}\sigma\derivada{}{}{N}(\sigma\nablafront{}\phi_k\cdot\nablafront{}v)-\sigma^2H\nablafront{}\phi_k\cdot\nablafront{}v.\nonumber
\end{eqnarray}
Now
\begin{eqnarray}
\derivada{}{}{t}\left(\nablafront{t}\phi_k\cdot\nablafront{t}v\right)\Big|_{t=0}& = &\left(\derivada{}{}{t}\nablafront{t}\phi_k\Big|_{t=0}\right)\cdot\nablafront{}v+\left(\derivada{}{}{t}\nablafront{t}v\Big|_{t=0}\right)\cdot\nablafront{}\phi_k\nonumber \\
& = &\left[\nablafront{}\left(\derivada{}{\phi_k}{t}\right)-\derivada{}{}{t}\left(\derivada{}{\phi_k}{N}\right)N-\derivada{}{\phi_k}{N}\dot{N}\right]\cdot\nablafront{}v+\nonumber \\
& + & \left[\nablafront{}\left(\derivada{}{v}{t}\right)-\derivada{}{}{t}\left(\derivada{}{v}{N}\right)N-\derivada{}{v}{N}\dot{N}\right]\cdot\nablafront{}\phi_k\nonumber \\
& = &\nablafront{}\dot{v}\cdot\nablafront{}\phi_k\nonumber.
\end{eqnarray}
It follows that 

\begin{eqnarray}
II
& = &
-\int_{\front{}}\left[ \derivada{}{\sigma}{t}+\sigma\derivada{}{\sigma}{N}+\sigma^2H\right]\nablafront{}\phi_k\cdot\nablafront{}v+\sigma^2\derivada{}{}{N}(\nablafront{}\phi_k\cdot\nablafront{}v)\nonumber \\
&   &-\int_{\front{}}\sigma\nablafront{}\phi_k\cdot\nablafront{}\dot{v}\nonumber.
\end{eqnarray}
 For the first term in the integral  $I$, we have
  $$\int_{\front{}}\phi_k\divfront{}(\sigma\nablafront{}\dot{v})=-\int_{\front{}}\sigma\nablafront{}\phi_k\cdot\nablafront{}\dot{v}.$$
Thus
\begin{eqnarray}
\int_{\front{}}\phi_k\derivada{}{\ddot{v}}{N}& = & -\int_{\front{}}2\sigma(\nablafront{}\phi_k\cdot\nablafront{}\dot{v}-\lambda_0\phi_k\dot{v}-\dot{\lambda}\phi_kv)\nonumber\\
&   &-\int_{\front{}}\sigma^2\derivada{}{}{N}(\nablafront{}\phi_k\cdot\nablafront{}v-\lambda_0\phi_kv)\nonumber \\ &   &-\int_{\front{}}\left[\derivada{}{\sigma}{t}+\sigma\derivada{}{\sigma}{N}+\sigma^2H\right](\nablafront{}\phi_k\cdot\nablafront{}v-\lambda_0\phi_kv)\nonumber
\end{eqnarray}
Recalling that  $\dot{v}\!|_{t=0}=\sum_{j=i}^m c_j\dot{\phi}_j+\dot{c}_j\phi_j$, $v=\sum_{j=1}^m c_j\phi_k$ and
$$
\ddot{\lambda}c_k+2\dot{\lambda}\dot{c}_k=-\int_{\front{}} \phi_k\derivada{}{\dot{v}}{N},
$$
we conclude that the possible values of  $\ddot{\lambda}$
 are given by the following equations in  $\mathbb{R}^m$:

$
\begin{array}{l}
(\ddot{\lambda}I+\stackrel{\circ\circ}{M})c+2(\dot{\lambda}I+\stackrel{\circ}{M})\dot{c}\;=0\\
(\dot{\lambda}I+\stackrel{\circ}{M})c=0
\end{array}
$
\linebreak
 where  $\stackrel{\circ}{M}$ was given above and 
\begin{eqnarray}
\stackrel{\circ\circ}{M}_{j,k}& = &\int_{\front{}}2\sigma \dot{Q}_{jk}+\sigma^2\derivada{}{}{N}Q_{jk}+\left[\derivada{}{\sigma}{t}+\sigma\derivada{}{\sigma}{N} + H\sigma^2\right]Q_{jk},\nonumber
\end{eqnarray}
\linebreak
$
\begin{array}{l}
Q_{jk}=\nablafront{}\phi_j\cdot\nablafront{}\phi_k-\lambda_0\phi_j\phi_k\\
\dot{Q}_{jk}=\nablafront{}\phi_k\cdot\nablafront{}\dot{\phi_j}-\dot{\lambda}\phi_k\phi_j-\lambda\phi_k\dot{\phi_j}.
\end{array}
$
\end{proof}

%\begin{obs}
%Note que utilizando apenas a expressão para a condição de fronteira encontrada em \ref{fronteira2} não é possível ver que as derivadas $\umpt{\lambda}, \dpt{\lambda}$ satisfazem uma equação em $\mathbb{R}^m$ inteiramente análoga a encontrada para as derivadas das curvas de autovalores do problema de Laplace com condição de Dirichlet. 
%\end{obs}

\begin{obs}\label{oed}
 It is not difficult to see that the matrix  $\stackrel{\circ\circ}{M}_{j,k}$
is symmetric. This will be important in the sequel.

\end{obs}
   
%Não faremos uso do corolário acima no restante deste trabalho, mas entendemos que, de certa forma, ele completa o teorema anterior além de servir para referências futuras.
 
% Como abordaremos o problema de genericidade dos autovalores para o pro\-ble\-ma de Robin necessitaremos de curvas de autovalores e autofunções que sejam pelo menos diferenciáveis. Mas como já vimos anteriormente não é possível utilizar o mesmo argumento do teorema anterior, contudo se supormos que a multiplicidade não se altera quando consideramos uma perturbação especifica $h(t,.)$ é possível obter uma curva regular de autovalores (com multiplicidade constante) se as função $\beta$ e $c$ também forem suficientemente regulares. Mais precisamente temos o seguinte.
% \begin{teo}
% Considere o problema (\ref{plr}) com funções $c$ de classe $\mathcal{C}^2$ e $\beta$ de classe $\mathcal{C}^3$. Sejam $\lambda_0$ um autovalor para o problema (\ref{plr}) e uma família de difeomorfismos $h(t,.)$ de classe $\mathcal{C}^3$ com a seguinte propriedade: dado $\epsilon>0$, exite $\delta>0$ tal que para todo $|t|<\delta$ o problema $(\ref{plrp})_{h(t,.)}$ possua um único autovalor $\lambda(t)$ com multiplicidade $m$ e $|\lambda(t)-\lambda_0|<\epsilon$. 
% \end{teo} 

\section{Multiplicity of the eigenvalues on symmetric domains} \label{SMA}

In this section, we discuss some consequences of  the symmetry on the multiplicity of the eigenvalues of  problem (\ref{eln}). If $G$ is a compact subgroup of
 the orthogonal group $O(n)$,  we say that $\Omega $ is \textit{$G$-symmetric} (or, it is\textit{ $G$-invariant}, or it  has \textit{symmetry $G$}) if    $g\Omega=\Omega$ for all  $g\in G$.  Let 
 $$ 
 \mathrm{Diff}_G^m(\Omega) = \{ h \in  \mathrm{Diff}^m(\Omega)  | \; h \circ g = gh,
\textrm{ for any } g \in G  \}.  
$$  If $\Omega $ is $G$-symmetric and $h \in \mathrm{Diff}_G^m(\Omega)$  then clearly
 $h(\Omega)$ is also $G$-symmetric and we can then restrict the topology defined in section
\ref{secperturb} to the set of $G$-symmetric regions.

% O capítulo esta organizado da seguinte forma: seção \ref{pf} contem o calculo diferencial desenvolvido por Henry, a seção \ref{pa} apresenta os resultados da Teoria de Representações de Grupos por fim a seção \ref{cs} contém as principais consequências de se estudar os autovalores do problema de Laplace em regiões simétricas, além de alguns resultados técni\cos necessários ao longo dos capítulos posteriores.
\subsection{Algebraic preliminaries}\label{pa}

We now  present some definitions and results from the Representation Theory of Compact Groups 
 (see \cite{rh} chapter 3, section  27 for details and proofs) that will be used in the sequel.

Let  $G$ be a compact group.  A \textit{representation} of  $G$  in a Hilbert space  $H$ is a group  homomorphism  $V:G\rightarrow GL(H)$, where
 $GL(H)$  is the group 
    (under composition)  of invertible  continuous  linear operators in $H$. 
     If $H$ is a complex (resp. real) Hilbert space  the 
    representation $V$  
  is called \textit{unitary} (resp. \textit{orthogonal})  if   
   the image $V(g)$, which we denote in the sequel by $V_g$, is an  
  unitary (resp. orthogonal) operator, for any $g \in G$.   
 
\begin{defi}
A representation $G$ is strongly continuous if  $\displaystyle\lim_{x\rightarrow e}V_x\xi=\xi$ for any $\xi\in H$.
\end{defi}
\begin{defi}
Let  $V:G\rightarrow GL(H)$ and  $V':G\rightarrow GL(H')$ be continuous representations of $G$. We say that
\begin{enumerate}
\item  $V$  and $V'$  are equivalent if there exists a linear isometry  $T:H\rightarrow H'$ such that  $V_x'\circ T=T\circ V_x$, for any  $x\in G$.
\item $V$ is finite dimensional  if  $H$ is finite dimensional.  
\item A closed subspace  $H_1\subset H$ is invariant for  $V$ if
 $V_xH_1\subset H_1$ for any  $x\in G$. The representation  $V':G\rightarrow GL(H_{1})$ is called a sub-representation of  $V$ and will be denoted by $V_{|H_1}$.
\item V  is irreducible if its  only closed invariant subspaces are  $\{0\}$ 
 and $H$. Otherwise, $V$ is called reducible.
\item If  $H=H_1\oplus H_2\oplus ...\oplus H_m$, where the  $H_i$ are invariant
  under  $V$, we write  $V=V_{|H_1}\oplus V_{|H_2}\oplus ...\oplus V_{|H_m}$ and
 say that  $V$ is a direct sum of the representations  $V_{|H_i}$.
\end{enumerate}
\end{defi}

\begin{teo}
Any irreducible unitary (resp. orthogonal)  representation of a compact group 
 $G$ is finite dimensional. $G$ is abelian if and only if all its irreducible representations have  dimension (complex) $1$.
\end{teo}
 Let  $V$ be a finite dimensional representation of 
  $G$.  The function  $\chi_{V}$ given by  $g\rightarrow trV_g$, where $tr$
 is the trace of the operator   $V_g$,
  is called  the 
  character of $V$.  Clearly, two equivalent representations have the 
 same character. 

   Let  $G$ be a  compact subgroup of the orthogonal group $O(n)$. 
 The set of all equivalent classes of continuous irreducible  representations 
 of $G$ is called the \textit{dual object} of $G$ and is  denoted by $\hat G$. 
 We denote by $ {\mathcal{X}}_{\sigma} $ the character of any representation in the class 
 $\sigma \in \hat G$ and by $ d_{\sigma}$ its dimension.  
   If $H$ is a Hilbert space  and $V:G \mapsto L(H)$ is a continuous 
 orthogonal representation of $G$, we can define, for each $\sigma \in \hat G$,  
 the operator $P_{\sigma} $ in $H$ by 
 \[ \langle P_{\sigma} \xi \:,\:\eta \rangle = 
\int_G \langle V_x \xi \:,\:\eta \rangle d_{\sigma} 
\mathcal{X}_{\sigma}(x)\,\, dx \] 
  $P_{\sigma}$ is a continuous projection (see \cite{ta}). We set $M_{\sigma}:= P_{\sigma}H$. 
 
   The following decomposition theorem will be important in the sequel. 
 A proof for \textit{unitary} representations can be found in \cite{rh}. For real 
 spaces it can be obtained from this result by complexification (see
 \cite{ta}).
\begin{teo}\label{td}
Let  $G$ be a compact subgroup of  $O(n)$ and  $V$ a continuous (unitary)  orthogonal  
 representation of $G$  in $H$.   
 For every  $\sigma \in$ $\hat G$,  let  $P_{\sigma}$ be the operator in  $H$   
 defined by 
\[ \langle P_{\sigma} \xi \:,\:\eta \rangle = 
\int_G \langle V_x \xi \:,\:\eta \rangle d_{\sigma} 
\mathcal{X}_{\sigma}(x)\, dx. \]     
 Then $P_{\sigma}$  is a projection operator in $H$.    
\newline  If   $\sigma \neq \sigma ' $ then  $M_{\sigma}$ and  $M_{\sigma'}$  
 are    orthogonal subspaces of  $H$,  
  $H$ =     $ \bigoplus_{ \sigma \in \hat{G}} M_{\sigma}$.  
  \newline     
  For each  $\sigma \in \hat{G} , M_{\sigma} $ is  either $\{ 0 \}$ or  a direct 
 sum of  $m_{\sigma} $  pairwise orthogonal, $d_{\sigma}$-invariant subspaces 
 $ L_{\sigma , j}$, on each of which    
    $ V_{\mid L_{\sigma,j}}  \in \sigma $. \newline     
The cardinal number  $m_{\sigma}$ may be finite or infinite. 
  \newline     
The subspace  $M_{\sigma}$ is the smallest closed subspace of  $H$ containing  
 all invariant subspaces of  $H$ on which  $V$  is in the class     $\sigma$. \newline     
This direct sum decomposition of $V$ is unique in the following sense. If     
\[   H \:=  \: \bigoplus_{\lambda \in \Lambda} N_{\lambda}, \]     
where each  $N_{\lambda}$ is an invariant subspace on which  $V$ is irreducible, 
then        
     
\[ \{\oplus N_{\lambda} \:| \: V_{\mid N_{\lambda}} \in \sigma \}\:= \: M_{\sigma}\]      
and there are  $m_{\sigma}$  subspaces $N_{\lambda}$ on each of which      
$V_{\mid N_{\lambda}} \in \sigma.$     
\end{teo}

 \subsection{Consequences of  the symmetry}\label{cs}
%\subsection{Decomposição de $L^2(\Omega)$ e Espaços invariantes}
 We now apply the abstract results of the previous section to derive some results on the multiplicity of the eigenvalues of  (\ref{eln}).  The main result
  was obtained   in  \cite{as} and \cite{ta}, for the Dirichlet Laplacian. The proof in the Neumann case is completely similar but is presented here for
 completeness.

 % . Supondo que $G$ possua um ponto livre, A.L.Pereira mostrou que em regiões $\Omega$ $G$-simétricas sempre existem autovalores múltiplos para o Laplaciano, exceto para subgrupos $G$ isomorfos a $\mathbb{Z}_2\oplus...\oplus\mathbb{Z}_2$ ($m$ vezes).
    
 If  $G$ is a compact subgroup of  $O(n)$, the ``natural'' action of  $G$ in  $\mathbb{R}^n$ is given by  $(g,x)\mapsto gx$. The subgroup 
 $G_x=\{g\in G: gx=x\}$ is called the \textit{isotropy group} of  $x\in\mathbb{R}^n$  and
  $G(x)=\{gx: g\in G\}$ is the \textit{orbit of  $x$} under this action. A point 
  $x\in\mathbb{R}^n$ such that $G_x=Id$ is called a free point for the action.

Let  $\Omega\subset\mathbb{R}^n$ be   open, bounded,    $G$-invariant  and
 $\Gamma:G\rightarrow GL(L^2(\Omega))$ the quasi-regular representation of  $G$
  
$$
\Gamma_gu=u\circ g^{-1},\,\,\forall\, g\in G, \,\,\forall\, u\in L^2(\Omega).
$$
 This representation is orthogonal and commutes with the Laplacian, that is
$$
(\Gamma_g\circ\Delta)u=\Gamma_g(\Delta u)=(\Delta u)\circ g^{-1}=\Delta(u\circ g^{-1})=(\Delta\circ\Gamma_g)u
$$
for any  $u\in H^2(\Omega)$, and  $g\in G$. As an immediate consequence the eigenspaces are invariant under the representation  $\Gamma$.

For any  $\sigma\in\hat{G}$, let $P_{\sigma}$ be the projection
$$
\left\langle P_{\sigma}f, h\right\rangle=\int_{G}\left\langle \Gamma_g f, h \right\rangle d_{\sigma}\chi_{\sigma}dg.
$$
Theorem  \ref{td} asserts  that  
$$
L^2(\Omega)=\displaystyle\bigoplus_{\sigma\in\hat{G}}M_{\sigma},
$$
where  $M_{\sigma}=P_{\sigma}L^2(\Omega)$.

 The  spaces $M_{\sigma}$ are invariant 
 for the Laplacian. More precisely
\begin{prop}\label{pi}
 If $
{\cal{D}}_N=\{u\in H^2(\Omega) | \frac{\partial u}{\partial N}=0,\,\,\textrm{on} \, \, \partial\Omega\}.
$   
 then the Laplacian is a linear transformation from  $M_{\sigma}\cap{\cal{D}}$ to  $M_{\sigma}$.
\end{prop}

Furthermore, we have

%\end{proof}
%  Using this  result, it can be proved  that the dimension of the eigenspaces
% of the Laplacian restricted to  $M_{\sigma}$, is a multiple of   $d_{\sigma}$. This follows immediately from the following   

\begin{prop}\label{dM}
 Each symmetry space $M_{\sigma}$ can be decomposed as a direct sum of subspaces
  $M_{\sigma}^i$ satisfying
\begin{enumerate}
\item $M_{\sigma}^i$  is invariant under the representation  $\Gamma$  and  $\Gamma_{M_{\sigma}^i}$ is an irreducible representation in the class  $\sigma$.
\item $M_{\sigma}^i$ is invariant for the Laplacian and   $\Delta|_{M_{\sigma}^i}$ is a multiple of the 
 identity, that is, the elements of  $M_{\sigma}^i$  are eigenfunctions associated to
 the same eigenvalue  $\lambda$. 
\end{enumerate}
\end{prop}  
\begin{proof}
Consider the spectral decomposition of the Laplacian restricted to  $M_{\sigma}$, that is, $M_{\sigma}=\oplus V_{\lambda}$, where  $V_{\lambda}$ is the eigenspace associated to the eigenvalue  $\lambda$. Since the Laplacian commutes with  $\Gamma$, the eigenspaces  $V_{\lambda}$ are invariant for the representation. From  Theorem \ref{td}  we have the decomposition
 $V_{\lambda}=V^1_{\lambda}\oplus...\oplus V^k_{\lambda}$, where each  $V^j_{\lambda}$ is an irreducible space in the class  $\sigma$.
 This proves the result.  

\end{proof}
\begin{coro}\label{amr}
 The multiplicity of each eigenvalue of the Laplacian restricted to 
 $M_{\sigma}$ is a multiple of the irreducible representations in the class $\sigma$.
\end{coro}
\begin{proof}
It follows immediately from proposition \ref{dM}. 

\end{proof}

Up to now, nothing precludes the possibility of  the spaces  $M_{\sigma}$ being trivial.
 For this, we need an additional technical condition.
  % Mesmo que $G$ contenha representações irredutíveis com dimensão $d_{\sigma}>1$ não podemos concluir que existam autovalores $\lambda$ para o Laplaciano com multiplicidade maior que 1. Entretanto, considerando uma hipótese técnica adicional ao grupo $G$, pode-se garantir que a representação quase regular $\Gamma$ contenha todas as classes de representações irredutíveis de $G$ e portanto garantindo que nenhum dos espaços $M_{\sigma}$ é vazio.  

\begin{teo}\label{tam}
If  $G$ is a compact subgroup of  $O(n)$ and there exists a free point  $x\in\Omega$ 
 under the natural action then, for each  $\sigma\in\hat{G}$ there is an eigenvalue
 $\lambda$ of the Laplacian, and a subspace  $H$  of the associated eigenspace
  $V_{\lambda}$ such that  $\Gamma_{|_H}$ is in the class  $\sigma$. In particular, for any  $\sigma\in\hat{G}$, there exist  an infinite number of eigenvalues
 whose multiplicity is a multiple of the dimension $d_{\sigma}$.
\end{teo}
\begin{proof}
 The result follows immediately from Corollary \ref{amr}, once it is known that
 the spaces  $M_{\sigma}$ are all infinite dimensional. This is proved in 
 \cite{ta} (Theorem. 3.2).

\end{proof}

 As an immediate consequence, we also obtain the following result.
\begin{coro}\label{cam}
If  $G$  is not a direct sum of cyclic groups of order $2$,  $\Omega$ is
 $G$-symmetric and  contains a free point under the action of $G$, then there always exist multiple eigenvalues of the Neumann Laplacian in $\Omega$.
\end{coro}

\section{Generic $G$-simplicity of the eigenvalues}\label{GSGA}
In this section, we analyze the validity of Conjecture \ref{conject1} for the Neumann Laplacian in the case of  \textit{finite groups}.
  We establish the validity of part I of the conjecture for arbitrary finite
subgroups $G$ of $O(n)$. Part II of the conjecture will be proved under an  additional assumption  on the dimension of the irreducible representations of $G$.
    %Em virtude de dificuldades técnicas, os casos $G$ finito e infinito serão tratados separadamente. De fato, ao  considerarmos $G$ infinito será necessário fazer uma hipótese adicional sobre o grupo, a saber, $dimG<n-1$ e $\Omega$ deve conter um ponto livre pela ação de $G$. Assim estabeleceremos a etapa I da conjectura para subgrupos infinitos compactos $G$ de $O(n)$. Além disso, também veremos que a conjectura é válida para subgrupos finitos que não possuam representações irredutíveis de dimensão maior que $2$. 

 An important step in our proof will be the analysis of the behavior of the eigeinvalues in each symmetric space. Here,  in contrast to the Dirichlet case analyzed in \cite{ta}, the knowledge of the first derivative of the eigenvalues did not suffice to separate multiple eigenvalues and  it became necessary to compute also the
 second derivative.  
%Outro aspecto que também será analisado na presente seção é o comportamento dos autovalores $G$-simples com relação a pequenas perturbações do domínio que preservam a simetria. Como podemos notar para os autovalores do problema de Laplace com condição de Dirichlet a análise da primeira derivada da curva de autovalores $G$-simples é suficiente para verificarmos que os autovalores $G$-simples variam com relação a perturbações simétricas do domínio. A situação é um pouco mais delicada para o problema com condição de Neumann, principalmente se o subgrupo $G$ for infinito. 

%%Na expressão da segunda derivada de autovalores aparece uma espécie de operador pseudodiferencial de fronteira que, de fato, fornecerá  mais informações sobre as autofunções na fronteira de $\Omega$. Neste ponto utilizaremos o método das soluções rapidamente oscilantes desenvolvido por Henry em \cite{henry} .
%%

%O estudo do comportamento das derivadas dos autovalores na presença de simetria é central em nossa abordagem. De forma geral  argumentaremos por contradição na demonstração dos teoremas. 

\subsection{A special case}

In this section we consider the very special case   where the symmetry group $G$ is isomorphic to    $ \mathbb{Z}_2\oplus\mathbb{Z}_2\oplus...\oplus\mathbb{Z}_2$ ($m$ times ).

 We first prove a technical result due to Uhlenbeck.

\begin{lema}\label{boundexp} 
Suppose  $\Omega\subset\mathbb{R}^n$ is  an open,  bounded,   $\mathcal{C}^2$-regular domain 
  $\lambda$ is a positive real number and   $f,g$ are  $\mathcal{C}^2$ functions on 
 $\partial\Omega$. If 
$$
 \nabla_{\partial\Omega} f \cdot \nabla_{\partial\Omega} g-\lambda fg=0,\, on\,\, \partial\Omega.
$$
 Then, at least one of those functions vanishes on $\partial\Omega$.
\end{lema}
 
\begin{proof} 
Let   $x(t)$ be a solution of the the equation   $\nabla_{\partial\Omega}f(x(t))=\stackrel{.}{x}(t),
 \ x(0)=x_0 \in \partial \Omega$. Since  $\partial\Omega$ is compact 
 $x(t)$ is defined for  $t$ and   $\frac{d}{dt}f(x(t))=|\nabla f(x(t))|^2\geq0$.
 Now, the function  $g(x(t))$ satisfies the equation  $\dot{u}(t)=\lambda f(x(t))u(t),
 \ u(0)= g(x_0)$ and, thus $g(x(t))=g(x_0)exp(\lambda\int_{0}^tf(x(s))ds)$.  Therefore, if
 $f(x_0) \neq  0$ and  $g(x_0)\neq0$, then  $g(x(t))$ would be unbounded   which cannot occur since   $\partial\Omega$
is compact 
\end{proof}

 \begin{teo}\label{tsalnz2} Suppose  $G$ be a subgroup of  $O(n)$ which is  isomorphic to  $\mathbb{Z}_2\oplus\mathbb{Z}_2\oplus...\oplus\mathbb{Z}_2$ and  $\Omega\subset\mathbb{R}^n$  an open,  bounded,   $\mathcal{C}^3$-regular,  $G$-symmetric domain. If  $\lambda_0$ is an eigenvalue of (\ref{eln})
with multiplicity  $m>1$ then, given $\epsilon>0$ there exist  $\delta>0$  and
 $h\in  Diff^3_{G}(\Omega) $, $||h-i_{\Omega}||_{\mathcal{C}^3}<\epsilon$ such that the
eigenvalues of  (\ref{elnp}) in the interval   $(\lambda_0-\delta,\lambda_0+\delta)$
 are all simple. 
\end{teo}
\begin{proof}
 Suppose  $\lambda_0$ is an eigenvalue of  (\ref{elnp}) with multiplicity
   $m>1$. It is enough to show that it  can be separated by small perturbations preserving the symmetry.  
 If $ h\in  Diff^3_{G}(\Omega) $, 
the perturbed problem in the Lagrangean form is
\begin{equation}\label{elnp}
                 \left\{
                       \begin{array}{lccc}
                             h^*(\Delta+ \lambda)h^{*-1}u=0,& in \ \ \Omega;\\
                              h^*\frac{\partial }{\partial N}h^{*-1}u=0, & \ \ on\ \ \partial \Omega;\\
                              
                       \end{array}
                            \right.
\end{equation} 

 If we choose an analytic family of diffeomorphism $t\rightarrow h(t,.)\in \mathcal{C}^3$, 
Theorem \ref{tecapln} guarantees the existence of   $m$ corresponding analytic curves of eigenvalues with derivatives given by the eigenvalues of the matrix
 (see corollary \ref{ed}) 
$$
\stackrel{\circ}{M}_{ij}=\int_{\partial\Omega}\sigma(\nabla\phi_i \cdot \nabla\phi_j-\lambda_0\phi_i\phi_j).
$$

 Suppose, by contradiction, that $\lambda_0$, cannot be split into eigenvalues of smaller multiplicity.  Then $ \stackrel{\circ}{M} $ must be   a multiple of the identity, that is
\begin{eqnarray}
& & \int_{\partial\Omega}\sigma(|\nabla\phi_i|^2-\lambda_0\phi_i^2-(|\nabla\phi_j|^2-\lambda_0\phi_j^2))=0 \label{dp}\\ 
& & \int_{\partial\Omega}\sigma(\nabla\phi_i \cdot \nabla\phi_j-\lambda_0\phi_i\phi_j)=0,\, i\neq j.\label{nd}
\end{eqnarray}
Since the family of diffeomorphism can be arbitrarily chosen in $ Diff^3_{G}(\Omega)$,
 the function  $\sigma$ can be any     $G$-invariant function on $\partial \Omega$.

 Let  $L^2(\Omega)=\bigoplus_{\chi\in\hat{G}} M_{\chi}$
 be the decomposition given by   Theorem \ref{td}. In the present case,
   $M_{\chi}=\{f\in L^2(\Omega): f\circ g=\chi(g)f,\, \forall g\in G\}$ 
 and  $\chi(g) \in\{-1,1\}$ for all  $g\in G$. Furthermore, we can choose an orthonormal basis of  the eigenspace  $V_{\lambda_0}$
 $\{\phi_j\}_{j=1}^{m}$, with $\phi_j\in M_{\chi_j}$ (the spaces $M_{\chi_j}$  need not be all
 distinct). We need to analyze two  situations

 case  i)  there exist more than one eigenfunction in 
 the same symmetry space  $M_{\chi}$. Thus, the expression  $\nabla\phi_i \cdot \nabla\phi_j-\lambda_0\phi_i\phi_j$ is a $G$-invariant function on  $\partial\Omega$. From    (\ref{nd})
 it follows that   $\nabla\phi_i \cdot \nabla\phi_j-\lambda_0\phi_i\phi_j=0$ on $\partial\Omega$, which cannot occur by Lemma \ref{boundexp}.

 case ii) happens,  there exist two eigenfunctions  $\phi_i,\phi_j$ belonging to 
 distinct symmetry spaces.  Since the functions   $|\nabla\phi_i|^2-\lambda_0\phi_i^2$ are  $G$-invariant for each $i$, it follows from  (\ref{dp}), that 
$$
\nabla(\phi_i+\phi_j)\cdot\nabla(\phi_i-\phi_j)-\lambda_0(\phi_i+\phi_j)(\phi_i-\phi_j)=|\nabla\phi_i|^2-\lambda_0\phi_i^2-(|\nabla\phi_j|^2-\lambda_0\phi_j^2)=0
$$
on  $\partial\Omega$. Writing  $\psi^+=\phi_i+\phi_j$ e $\psi^-=\phi_i-\phi_j$, we have $\nabla\psi^+ \cdot \nabla\psi^--\lambda_0\psi^+\psi^-=0$, where $\psi^+$ e $\psi^-$ 
 are eigenfunctions associated to  $\lambda_0$, again in contradiction with Lemma
 \ref{boundexp}.
\end{proof}

\begin{coro}
If  $G$ is a finite subgroup of $O(n)$ isomorphic to  $\mathbb{Z}_2\oplus\mathbb{Z}_2\oplus...\oplus\mathbb{Z}_2$ then, for a residual set of open, bounded,  connected regions
 $\mathcal{C}^2$-regular $G$-symmetric  regions    $\Omega$ of  $\mathbb{R}^n$, the eigenvalues of  (\ref{eln}) are all simple. 
\end{coro}
\begin{proof}
Let 
\begin{eqnarray}
& & {\cal{C}}_k=\{h\in Diff^3_{G}(\Omega): {\mathrm{the\, eigenvalues , \lambda\, of \,\, \,\, (\ref{eln})\,,}}\nonumber\\
& & \, \lambda<k,\,{\mathrm{ are\, all \,\, simple}} \}.\nonumber
\end{eqnarray}
 ${\cal{C}}_k$ is open by the continuity properties asserted by Theorem \ref{tcaln}. 
Theorem \ref{tsalnz2} guarantees that ${\cal{C}}_k$  is also dense. The result then follows by 
 taking intersection in $k$.

\end{proof}

\subsection{General finite groups} \label{secfiniteg}

 We now consider the problem for a general finite group $G$. As we will see,
 the first part of conjecture  \ref{conject1} (sub-conjecture I) can be established in this general case (though the arguments are  more involved than the Dirichlet case).  However the second part is much more difficult and we have only
been  able to establish it in some special cases. In fact,  even
 in the first step and supposing the eigenvalues do not split, 
 the expression of the first derivative of the eigenvalues,
 given by the matrix  $\stackrel{\circ}{M}$ does not suffice to obtain a contradiction.  Therefore we are forced to compute the second derivative. Then, the hypothesis of non separability implies that a certain boundary operator is of finite range. At this point, we use the ``Method of Rapidly Oscillating Solutions''
 (see section \ref{OF}) to obtain more information on the eigenfunctions, which finally lead to the searched for contradiction.  
\begin{teo}\label{tsalnme}
 Let  $G$ be a finite subgroup of  $O(n)$ and  $\Omega\subset\mathbb{R}^n$ and open bounded, connected,  $\mathcal{C}^3$-regular  $G$-invariant domain. Let also  $\lambda_0$ be an eigenvalue with multiplicity   $md_{\sigma}$, $m>1$, which is the unique eigenvalue  for the problem  (\ref{eln}) restricted to  $M_{\sigma}$ in the 
 interval  $(\lambda_0-\delta,\lambda_0+\delta)$ . Given  $\epsilon>0$ there exists   $h\in Diff^3_{G}(\Omega)$, $||h-i_{\Omega}||_{\mathcal{C}^3}<\epsilon$ such that the problem   (\ref{elnp}) restricted to
$ M_{\sigma}$ has exactly $m$ 
 \textit{$G_{\sigma}$-simple} eigenvalues in the interval  $(\lambda_0-\delta,\lambda_0+\delta)$.
\end{teo}   
\begin{proof}  Let  $\{\phi_j^i\}$, $i=1,...,m; j=1,...,d_{\sigma}$ be
 an orthonormal basis for the eigenspace associated to  $\lambda_0$
 satisfying 
	\begin{eqnarray}\label{invariancian}
                 \left(\begin{array}{c}
                           \phi_1^{i}\\
                            .\\
                            .\\
                         
                            \phi^i_{d_{\sigma}}\\   
             
                   \end{array} \right)
             \circ g = A_{\sigma}(g)
               \left(\begin{array}{c}
                           \phi_1^{i}\\
                            .\\
                            .\\
                          
                            \phi^i_{d_{\sigma}}\\   
                     \end{array}\right),              
           \end{eqnarray}      
for all  $g\in G$ where  $g\mapsto A_{\sigma}(g)$  is an irreducible
  matrix representation of dimension $d_{\sigma}$  in the class  $\sigma$.
 Consider the  renumbering of the  functions  $\phi^i_j$ given by,
  $\varphi_k=\phi^i_j$, where  $k=(i-1)d_{\sigma}+j$ , that is 
$$
\varphi_1=\phi^1_1,...,\varphi_{d_{\sigma}}=\phi^1_{d_{\sigma}},\varphi_{d_{\sigma}+1}=\phi^2_1,...,\varphi_{2d_{\sigma}}=\phi^2_{d_{\sigma}},...,\varphi_{md_{\sigma}}=\phi^m_{d_{\sigma}}. 
$$
 Suppose that the multiplicity of  $\lambda_0$ cannot be reduced by small G-symmetric perturbations of $\Omega$. Then, the matrix   
 $\stackrel{\circ}{M}$, given by 
$$
\stackrel{\circ}{M}_{lk}=\int_{\partial\Omega}\sigma(\nabla\varphi_l\cdot\nabla\varphi_k-\lambda_0\varphi_l\varphi_k)
$$
 is such that  $\stackrel{\circ}{M}=\stackrel{.}{\lambda} I$, that is 
\begin{eqnarray}
& & \int_{\partial\Omega}\sigma(|\nabla\varphi_k|^2-\lambda_0\varphi_k^2)=\int_{\partial\Omega}\sigma(|\nabla\varphi_l|^2-\lambda_0\varphi_l^2)) \label{dlnf}\\ 
& & \int_{\partial\Omega}\sigma(\nabla\varphi_k \cdot \nabla\varphi_l-\lambda_0\varphi_k\varphi_l)=0,\, 1\leq k, l\leq md_{\sigma}.\label{ndlnf}
\end{eqnarray}
It is difficult to obtain some information from this relations, since the integrands are not $G$-invariant. However, taking into account the renumbering above, we see that  the entries of the matrix 
 $\stackrel{\circ}{M}$  contain the  expressions   
$$
\nabla\phi_j^i \cdot \nabla\phi_{j}^l-\lambda_0\phi_j^i\phi_{j}^l
$$ for  $1\leq i,l\leq m$. We can obtain some new information, if we show that their sum  
$$
\sum_{j=1}^{d_{\sigma}}\nabla\phi_j^i \cdot \nabla\phi_{j}^l-\lambda_0\phi_j^i\phi_{j}^l
$$
%------
$$
\sum_{j=1}^{d_{\sigma}}|\nabla\phi_j^i|^2-\lambda_0(\phi_j^i)^2
$$
are  $G$-invariant functions  on  $\partial\Omega$. To this aim, we show that the sum involving the gradient is  $G$-invariant, since the other sum is clearly
 $G$-invariant. In fact,

$$
\phi_j^i\circ g^{-1}(x)=\sum_{k=1}^{d_{\sigma}}a_{j,k}(g)\phi_k^i,
$$
where $a_{jk}(g)$  are the entries in the matrix representation  $g\rightarrow A_{\sigma}(g)$.
It follows that 
\begin{eqnarray*}
\sum_{j=1}^{d_{\sigma}}(\nabla\phi_{j}^i\cdot\nabla\phi_j^l)(g^{-1}x)&=& \sum_{j,k,p}^{d_{\sigma}}a_{jk}(g)a_{jp}(g)\nabla\phi_k^i\cdot\nabla\phi_p^l(x)\\
& = &\sum_{k,p}^{d_{\sigma}}\delta_{kp}\nabla\phi_k^i\cdot\nabla\phi_p^l(x)\\
& =&\sum_{j=1}^{d_{\sigma}}\nabla\phi_{j}^i\cdot\nabla\phi_j^l(x).
\end{eqnarray*}
 The proof that  $\sum_{j=1}^{d_{\sigma}}|\nabla\phi_j^i|^2-\lambda_0(\phi_j^i)^2$   is  $G$-invariant in  $\partial\Omega$ is analogous.

 Therefore,  observing that the function   $\sigma$ can be chosen arbitrarily close to any $G$-invariant function on $\partial \Omega$, relations  (\ref{ndlnf}) e (\ref{dlnf}) give 
\begin{equation*}%\label{cpd2}
\sum_{j=1}^{d_{\sigma}}|\nabla\phi_j^i|^2-\lambda_0(\phi_j^i)^2=\sum_{j=1}^{d_{\sigma}}|\nabla\phi_j^l|^2-\lambda_0(\phi_j^l)^2
\end{equation*}

\begin{equation}\label{cpd}
\sum_{j=1}^{d_{\sigma}}\nabla\phi_j^i \cdot \nabla\phi_{j}^l-\lambda_0\phi_j^i\phi_{j}^l=0, \quad
 \textrm{on} \quad  \partial\Omega.
\end{equation}
 Even with this new information about the eigenfunctions in the boundary, we could not obtain a contradiction.
  We thus calculated the second derivative of the curve of eigenvalues, using corollary \ref{ed}\footnote{One can obtain the expression of the matrix of the second derivative without appealing to corollary  \ref{ed} since, supposing the non separability of the eigenvalues it is legitimate to take derivatives directly from the expression for the first derivative.} 

\begin{eqnarray*}
\stackrel{\circ\circ}{M}_{k,j}& = &\int_{\partial\Omega}2\sigma \dot{Q}_{jk}+\sigma^2\frac{\partial}{\partial N}Q_{jk}+\left[\frac{\partial\sigma}{\partial t}+\sigma\frac{\partial\sigma}{\partial N} + H\sigma^2\right]Q_{jk}
\end{eqnarray*}
where 
$$
\begin{array}{l}
Q_{jk}=\nabla_{\partial\Omega}\varphi_j\cdot\nabla_{\partial\Omega}\varphi_k-\lambda_0\varphi_j\varphi_k\\
\dot{Q}_{jk}=\nabla_{\partial\Omega}\varphi_k\cdot\nabla_{\partial\Omega}\dot{\varphi_j}-\dot{\lambda}\varphi_k\varphi_j-\lambda\varphi_k\dot{\varphi_j},
\end{array}
$$

 and  $\dot{\varphi}_j$  is the unique solution of 

\begin{equation}\label{eofcln}
\left\{
\begin{array}{lc}
(\Delta+\lambda_0)\dot{\varphi}_j\ \in span[\varphi_i]_1^{m d_{\sigma}}, \\
\frac{\partial\dot{\varphi}_j}{\partial N}=\nabla_{\partial\Omega}\sigma\cdot\nabla_{\partial\Omega}\varphi_j-\sigma\frac{\partial^2}{\partial N^2}\varphi_j,&  \ \ on\ \ \partial \Omega\\
\dot{\varphi}_j\bot span[\varphi_i]_1^{m d_{\sigma}}.
\end{array}
\right. 
\end{equation}
 In order to obtain $G$ invariant functions,  we will again need to sum up some entries of the matrix
 $\stackrel{\circ\circ}{M}$. Actually, we will see that the integrand of
  $\sum_{j=1}^{d_{\sigma}}\stackrel{\circ\circ}{M}_{j,j+d_{\sigma}}$ is  $G$-invariant. We know from  (\ref{cpd}) that, if the multiplicity cannot be reduced, then
$$
\sum_{j=1}^{d_{\sigma}}Q_{j,j+d_{\sigma}}=\sum_{j=1}^{d_{\sigma}}\nabla\phi_j^1 \cdot \nabla\phi_{j}^2-\lambda_0\phi_j^1\phi_{j}^2=0.
$$
From the $G$ invariance of $\sum_{j=1}^{d_{\sigma}}Q_{j,j+d_{\sigma}}$,  it follows that  $\sum_{j=1}^{d_{\sigma}}\frac{\partial}{\partial N}Q_{j,j+d_{\sigma}}$ is  $G$-invariant. From the symmetry of $\stackrel{\circ\circ}{M}_{j,k}$, we obtain  
$$
\int_{\partial\Omega}2\sigma \dot{Q}_{jk}=\int_{\partial\Omega}\sigma(\dot{Q}_{jk}+\dot{Q}_{kj}).
$$
Therefore, to show that the integrand of the expression  $\sum_{j=1}^{d_{\sigma}}\stackrel{\circ\circ}{M}_{j,j+d_{\sigma}}$ is also  $G$-invariant it is enough to show that   $\sum_{j=1}^{d_{\sigma}}\dot{Q}_{jj+d_{\sigma}}+\dot{Q}_{j+d_{\sigma}j}$  is  $G$-invariant. This follows from the fact that  $t \rightarrow\sum_{j=1}^{d_{\sigma}}Q_{j,j+d_{\sigma}}(t)$ is a  $\mathcal{C}^1$ curve in the space of   $G$-invariant functions.

 From the non separability of the eigenvalues, it follows that   $\sum_{j=1}^{d_{\sigma}}\stackrel{\circ\circ}{M}_{j,j+d_{\sigma}}=0$ for any  $G$-invariant $\sigma$ and, therefore  
\begin{equation}\label{csdln}
\sum_{j=1}^{d_{\sigma}}\dot{Q}_{jj+d_{\sigma}}+\dot{Q}_{j+d_{\sigma}j}+ \sigma\frac{\partial}{\partial N}Q_{j+d_{\sigma}j}=0.
\end{equation}
 To simplify the notation, we introduce the bilinear form  ${\cal{Q}}(u,v)=\nabla v\cdot\nabla u-\lambda_0vu$.  Then  (\ref{csdln}) can be rewritten as 
\begin{equation}\label{eqofln}
\sum_{j=1}^{d_{\sigma}}\sigma\frac{\partial}{\partial N}{\cal{Q}}(\phi_j^1,\phi_{j}^2)+{\cal{Q}}(\phi_j^1,\dot{\phi}_{j}^2)+{\cal{Q}}(\phi_j^2,\dot{\phi}_{j}^1)=\sum_{j=1}^{d_{\sigma}}\dot{\lambda}(\phi_j^1\phi_j^2).
\end{equation}

 The solutions  $\dot{\phi}^i_j = \dot{\varphi}_{(i-1)d_{\sigma}+j}$ of  (\ref{eofcln}) as functions of 
 $\sigma$   define a boundary operator which we denote by 
  ${\cal{C}}_j^i(\sigma)$.
  %(we prove in  \ref{OF} that  ${\cal{C}}_J^i$ is well defined ),
 Then,  equation  (\ref{eqofln}) defines  a boundary operator given by 
\begin{equation}\label{ofln1}
\Xi(\sigma)= \sum_{j=1}^{d_{\sigma}}\sigma\frac{\partial}{\partial N}{\cal{Q}}(\phi_j^1,\phi_{j}^2)+{\cal{Q}}(\phi_j^1,{\cal{C}}_j^2(\sigma))+{\cal{Q}}(\phi_j^2,{\cal{C}}_j^1(\sigma))
\end{equation}
where $\sigma$ is a  $G$-invariant function on  $\partial\Omega$.
 From (\ref{eqofln}), it follows that the operator  $\Xi$ is of finite range. A
 necessary condition for this  (theorem \ref{tofln1}) is that 
$$
\sum_{j=1}^{d_{\sigma}}\frac{\partial\phi_{j}^1}{\partial\tau}\frac{\partial\phi_{j}^2}{\partial\tau}=0
$$
 for any  $x \in \partial\Omega$ and $\tau\in T_{x}\partial\Omega$. 
 
 We can repeat the whole process  substituting $\phi^1_j$ by  
 $\phi^1_j\circ g$.
Looking at this relation as the inner product of vectors $v_1=(\frac{\partial \phi_1^1}{\partial\tau},\dots,\frac{\partial \phi_{d_{\sigma}^1}}{\partial\tau}), v_2=(\frac{\partial \phi_1^2}{\partial\tau},\dots,\frac{\partial \phi_{d_{\sigma}^2}}{\partial\tau})$ in   $\mathbb{R}^{d_{\sigma}}$, we have that  $\left\langle A_{\sigma}(g)v_1,v_2\right\rangle=0$ for all
  $g\in G$.  Since $A_\sigma$ are irredutible representation of the $G$, we have  $\frac{\partial\phi_{j}^2}{\partial\tau}=0$ for all  $x\in\partial\Omega$ and  $\tau\in T_{x}(\partial\Omega)$. It follows that  $\nabla\phi^2_j=0$  on  $\partial\Omega$. Therefore, using (\ref{cpd}), 
we obtain  $\sum_{j=1}^{d_{\sigma}}\phi_j^1\phi_{j}^2=0$  on  $\partial\Omega$.
 The process can be repeated again with  $\phi^1_j\circ g$ 
 in the place of  $\phi^1_j $, to obtain  $\phi_j^2=0$ on  $\partial\Omega$.
 Since   $\phi_j^2$ also satisfies  $\frac{\partial\phi_j^2}{\partial N}=0$ on  $\partial\Omega$,   Cauchy  Uniqueness Theorem assures that   $\phi_j^2\equiv0$ on  $\Omega$, which gives the desired contradiction.

\end{proof}
\begin{coro}\label{cp1}
Let  $G$  be a finite subgroup of  $O(n)$ and   $\sigma\in\hat{G}$.  Then the set
 
\[ {\cal{C}}= \{h\in Diff_G^2(\Omega)|  \, \textrm{ the eigenvalues of the problem 
 (\ref{eln}) restricted to  $M_{\sigma}$  are all  $G_{\sigma}$-simple }\} 
 \]
is residual in 
$Diff_G^3(\Omega)$.
%\begin{eqnarray}
%& &{\cal{C}}=\{h\in Diff_G^2(\Omega)|{\mathrm{\,\, all\, the  \, autovalores \, para \,\, o\,\, problema \, (\ref{eln}),\,\, restrito\,\, a}}\nonumber\\
%& & M_{\sigma},\,\,{\mathrm{s\tilde{a}o \,\, G_{\sigma}-simples\}}} \nonumber
%\end{eqnarray} 
\end{coro}
\begin{proof}
 Let 
 $${\cal{C}}_k = \{h\in Diff_G^2(\Omega)|  \, \textrm{ the eigenvalues of the problem 
 (\ref{eln}) restricted to  $M_{\sigma}$  are all  $G_{\sigma}$-simple }\}$$
 
 We prove  that $ {\cal{C}}_k$ is  open and dense and then take intersection for $ k \in \mathbb{N}$. 
To prove openness it is enough to observe that the  proof   of   continuity property of the eigenvalues  given in Theorem  \ref{tcaln} can be easily adapted to show the same properties   for the problem restricted to each symmetry space  $M_{\sigma}$. For the   density part, we can assume  more smoothness and then use  Theorem \ref{tsalnme} above. 

%\begin{eqnarray}
%& & {\cal{C}}_k=\{h\in Diff^3_{G}(\Omega): {\mathrm{os\, autovalores\, \lambda\,\, de\,\, (\ref{eln})\, restrito\,a }}\, M_{\sigma},\nonumber\\
%& & \, \lambda<k,\,{\mathrm{ s\tilde{a}o\, todos\,\, G-simles}} \}.\nonumber
%\end{eqnarray}

\end{proof}

We now consider the second part of conjecture  \ref{conject1} for finite groups. For this step, which involves the separation of eigenvalues in different spaces of symmetry  we will need an additional hypotheses on the dimension of irreducible representations of $G$.
 We start with a technical auxiliary result.

\begin{lema}\label{lsed}
    Let  $M$ be a differentiable manifold and  $F,G:M\longrightarrow\mathbb{R}^2$ differentiable functions.
 If    $|F(x)|=|G(x)|$ and $|\frac{\partial F}{\partial\tau}|=|\frac{\partial G}{\partial\tau}|$ for any $\tau\in T_xM$,
 then there exists an open set $V$ in $M$ and an orthogonal transformation   $T$  in  $\mathbb{R}^2$ such that $F(x)=TG(x)$ in  $V$.
\end{lema}

\begin{proof}
Using complex notation, we have  $F(x)=e^{i\theta(x)}G(x)$. If   $\theta(x)$ is constant in some open set, we are done. 
  Suppose then  that  $\nabla_{M}\theta(x)$ does not vanish identically in any open subset of $M$..  Choosing local coordinates  $(x_1,...,x_{n-1})$  in $M$  and 
 $\tau=\frac{\partial}{\partial x_i}+\frac{\partial}{\partial x_j}$, it follows from the condition
 $|\frac{\partial F}{\partial\tau}|=|\frac{\partial G}{\partial\tau}|$ that 
 
\begin{equation}\label{prffgg}
 	\begin{array}{ccc }
 		 Re(\frac{\partial F}{\partial x_i}\frac{\partial \stackrel{-}{F}}{\partial x_j}) & = & Re(\frac{\partial G}{\partial x _i}  \frac{\partial \stackrel{-}{G}}{\partial x_j}).
  \end{array}
 \end{equation}

%Como estamos supondo $F(x)=e^{i\theta(x)}G(x)$ temos,
Thus 
\begin{eqnarray}
	     & & \frac{\partial F}{\partial x_i}\frac{\partial \stackrel{-}{F}}{\partial x_j} = \left(i\frac{\partial\theta}{\partial x_i}e^{i\theta}G+e^{i\theta}\frac{\partial G}{\partial x_i}\right)\left(-i\frac{\partial\theta}{\partial x_j}e^{-i\theta}\stackrel{-}{G}+e^{-i\theta}\frac{\partial\stackrel{-}{G}}{\partial x_j}\right)\nonumber\\
& & = \frac{\partial\theta}{\partial x_i}\frac{\partial\theta}{\partial x_j}|G|^2+i\left(\frac{\partial\theta}{\partial x_i}G\frac{\partial\stackrel{-}{G}}{\partial x_j}-\frac{\partial\theta}{\partial x_j}\stackrel{-}{G}\frac{\partial G}{\partial x_i}\right)+\frac{\partial G}{\partial x_i}\frac{\partial\stackrel{-}{G} }{\partial x_j}.  \label{ffgg}
\end{eqnarray}

Now,
	\begin{eqnarray}
		Re\left(\frac{\partial\theta}{\partial x_i}\frac{\partial\theta}{\partial x_j}|G|^2+i\left(\frac{\partial\theta}{\partial x_i}G\frac{\partial}{\partial x_j}\stackrel{-}{G}-\frac{\partial\theta}{\partial x_j}\stackrel{-}{G}\frac{\partial }{\partial x_i}G\right)\right)=\nonumber\\
		\frac{\partial\theta}{\partial x_i}\frac{\partial\theta}{\partial x_j}|G|^2-Im\left(\frac{\partial\theta}{\partial x_i}G\frac{\partial}{\partial x_j}\stackrel{-}{G}-\frac{\partial\theta}{\partial x_j}\stackrel{-}{G}\frac{\partial }{\partial x_i}G\right).\label{prff} 
	\end{eqnarray}
 Writing $G=g_1+ig_2$
$$
G\frac{\partial}{\partial x_j}\stackrel{-}{G}=g_1\frac{\partial g_1}{\partial x_j}+g_2\frac{\partial g_2}{\partial x_j}+i\left(g_2\frac{\partial g_1}{\partial x_j}-g_1\frac{\partial g_2}{\partial x_j}\right)
$$
$$
\stackrel{-}{G}\frac{\partial}{\partial x_i}G=g_1\frac{\partial g_1}{\partial x_i}+g_2\frac{\partial g_2}{\partial x_i}+i\left(g_1\frac{\partial g_2}{\partial x_i}-g_2\frac{\partial g_1}{\partial x_i}\right)
$$
it follows that 
$$
Im\left(\frac{\partial\theta}{\partial x_i}G\frac{\partial}{\partial x_j}\stackrel{-}{G}-\frac{\partial\theta}{\partial x_j}\stackrel{-}{G}\frac{\partial }{\partial x_i}G\right)=\frac{\partial\theta}{\partial x_i}\left(g_2\frac{\partial g_1}{\partial x_j}-g_1\frac{\partial g_2}{\partial x_j}\right)+\frac{\partial\theta}{\partial x_j}\left(g_2\frac{\partial g_1}{\partial x_i}-g_1\frac{\partial g_2}{\partial x_i}\right).
$$
Taking the real part on identity
 (\ref{ffgg})  and using relations (\ref{prffgg}),  (\ref{prff}), we obtain  
\begin{equation}
		\frac{\partial\theta}{\partial x_i}\left(\frac{1}{2}\frac{\partial\theta}{\partial x_j}+ \frac{g_2\frac{\partial g_1}{\partial x_j}-g_1\frac{\partial g_2}{\partial x_j}}{|G|^2} \right)+\frac{\partial\theta}{\partial x_j}\left(\frac{1}{2}\frac{\partial\theta}{\partial x_i}+ \frac{g_2\frac{\partial g_1}{\partial x_i}-g_1\frac{\partial g_2}{\partial x_i}}{|G|^2} \right)=0\nonumber.	
\end{equation}
    Since we are assuming $\nabla_{\partial\Omega}\theta$ does not vanish in any open set, the same follows for   $g_2$.
   Thus    the above  equation can be rewritten as
\begin{equation}\label{aux1}
			\frac{\partial\theta}{\partial x_i}\left(\frac{1}{2}\frac{\partial\theta}{\partial x_j}+ \frac{\partial}{\partial x_j}\arctan\left(\frac{g_1}{g_2}\right) \right)+\frac{\partial\theta}{\partial x_j}\left(\frac{1}{2}\frac{\partial\theta}{\partial x_i}+ \frac{\partial}{\partial x_i}\arctan\left(\frac{g_1}{g_2}\right) \right)=0.
\end{equation} 
	Using again that $\nabla_{\partial\Omega}\theta$ does not vanish in any open set, at least one component  $\frac{\partial\theta}{\partial x_k}$
has the same property.  Taking  $i=j=k$ em (\ref{aux1}), we obtain 
$$
\frac{\partial}{\partial x_k}\left(\frac{\theta}{2}+\arctan\left(\frac{g_1}{g_2}\right)\right)=0,
$$
 in an open set.  Taking   $i=k$ in  (\ref{aux1}), the same identity follows for any index $j$.  Therefore $\theta=-2\arctan(\frac{g_1}{g_2})+C$ and 

\begin{eqnarray}
		F(x)& = & e^{i\left(-2\arctan(\frac{g_1}{g_2})+C\right)}G(x)=e^{iC}\frac{G(x)}{\left(e^{i\arctan(\frac{g_1}{g_2})}\right)^2}\nonumber\\
		    & = & e^{iC}\frac{G(x)}{\left(\frac{G(x)}{|G(x)|}\right)^2}=e^{iC}\stackrel{-}{G}(x). \nonumber
\end{eqnarray}
 Therefore, the orthogonal transformation  $T$ is given by 
$T=
	\left[\begin{array}{cc}
	       	\cos C & -\sin C\\
	      	\sin C & \cos C	      	        
	      \end{array}
	\right]
	\left[\begin{array}{cc}
	       	1 & 0\\
	      	0 &-1	      	        
	      \end{array}
	\right].
$

\end{proof}

\begin{teo}\label{tslned}
	Let   $G$ be a finite subgroup of $O(n)$ such that $d_{\sigma} \leq 2$ for any 
 $\sigma \in \hat{G}$ 
 and  $\Omega\subset\mathbb{R}^n$ an open  bounded connected   $\mathcal{C}^3$-regular and  $G$-symmetric domain.
  Suppose   $\lambda$ is the unique  eigenvalue for the problem  (\ref{eln}) restricted to the  symmetry spaces  $M_{\sigma_1}$  and  $M_{\sigma_2}$ in the interval  $(\lambda-\delta,\lambda+\delta)$. Suppose also that 
  the action of  $G$  in both  $ker(\Delta|_{M_{\sigma_1}}+\lambda)$  and   $ker(\Delta|_{M_{\sigma_2}}+\lambda)$ is
 irreducible.  Then, for any  $\epsilon>0$, there exists  $h\in Diff_{G}^{3}(\Omega)$, $||h-i_{\Omega}||_{\mathcal{C}^3}<\epsilon$ 
 and  $\delta>0$ such that there are exactly two  $\lambda_1(h), \lambda_2(h)$ $G$-simple eigenvalues 
  for the problem (\ref{elnp}) restricted to the space  $M_{\sigma_1}\oplus M_{\sigma_2}$ in the interval  $(\lambda-\delta,\lambda+\delta)$.  In other words, the natural action of  $G$ on  $ker(h^*\Delta h^{*-1}|_{M_{\sigma_2}\oplus M_{\sigma_1}}+\lambda_1(h))$ and  $ker(h^*\Delta h^{*-1}|_{M_{\sigma_2}\oplus M_{\sigma_1}}+\lambda_2(h))$ is irreducible.
\end{teo}
\begin{proof} Assume that the eigenvalue  $\lambda$ cannot be separated by small $G$-symmetric perturbations.
 Then the matrix of the first derivatives $\stackrel{\circ}{M}$, given by the Corollary \ref{ed}  must be a multiple of the 
 identity. Thus
\begin{eqnarray}\label{dp22}
& & \int_{\partial\Omega}\sigma(|\nabla\varphi_k|^2-\lambda_0\varphi_k^2)=\int_{\partial\Omega}\sigma(|\nabla\varphi_l|^2-\lambda_0\varphi_l^2)
\end{eqnarray}
 where  $\varphi_j=\phi^1_j$ if  $1\leq j\leq d_{\sigma_1}$, \    $\varphi_j=\phi^2_{j-d_{\sigma_1}}$ if  $d_{\sigma_1}+1\leq j\leq d_{\sigma_1}+d_{\sigma_2}$  and the eigenfunctions   $\{\phi_j^1\}_{j=1}^{d_{\sigma_1}}$ and  $\{\phi_j^2\}_{j=1}^{d_{\sigma_2}}$ satisfy  (\ref{invariancian}).
  As in the proof of Theorem \ref{tsalnme}, we build the $G$-invariant functions
\begin{eqnarray}
& & \sum_{j=1}^{d_{\sigma_1}}|\nabla\varphi_j|^2-\lambda(\varphi_j)^2= \sum_{j=1}^{d_{\sigma_1}}|\nabla\phi^1_j|^2-\lambda(\phi^1_j)^2, \nonumber\\
& & \sum_{j=1+d_{\sigma_1}}^{d_{\sigma_1}+d_{\sigma_2}}|\nabla\varphi_j|^2-\lambda(\varphi_j)^2=\sum_{j=1}^{d_{\sigma_2}}|\nabla\phi^2_j|^2-\lambda(\phi^2_j)^2.\nonumber
\end{eqnarray}
 It then  follows from  (\ref{dp22}) that
\begin{equation}\label{rpd}
\frac{1}{d_{\sigma_1}}\sum_{j=1}^{d_{\sigma_1}}|\nabla\phi^1_j|^2-\lambda(\phi^1_j)^2=\frac{1}{d_{\sigma_2}} \sum_{j=1}^{d_{\sigma_2}}|\nabla\phi^2_j|^2-\lambda(\phi^2_j)^2.
\end{equation}
Since we still cannot find a  contradiction, we proceed by computing the second derivative. Arguing   as in Theorem  \ref{tsalnme}, 
 we conclude that the boundary operator  
\begin{eqnarray}\label{ofln3}
\Phi(\sigma) & = & \frac{1}{d_{\sigma_1}}\sum_{j=1}^{d_{\sigma_1}}\sigma\frac{\partial}{\partial N}{\cal{Q}}(\phi^1_k,\phi^1_{k})-2({\cal{Q}}(\phi^1_k,{\cal{C}}_k^1(\sigma))\nonumber\\
             &   &-\frac{1}{d_{\sigma_2}}\sum_{j=1}^{d_{\sigma_2}}\sigma\frac{\partial}{\partial N}{\cal{Q}}(\phi^2_k,\phi^2_{k})-2({\cal{Q}}(\phi^2_k,{\cal{C}}_k^2(\sigma))
\end{eqnarray} 
is of finite range. It follows from  Theorem  \ref{tofln3} that 
\begin{equation}\label{rsd}
\frac{1}{d_{\sigma_1}}\sum_{j=1}^{d_{\sigma_1}}\left(\frac{\partial\phi^1_j}{\partial\tau}\right)^2=\frac{1}{d_{\sigma_2}}\sum_{j=1}^{d_{\sigma_2}}\left(\frac{\partial\phi^2_j}{\partial\tau}\right)^2
\end{equation}
for any  $\tau\in T_{x}(\partial\Omega)$.
 Thus $ \frac{1}{d_{\sigma_1}}  |\nabla\phi^1_j|^2 =
  \frac{1}{d_{\sigma_2}}  |\nabla\phi^2_j|^2 $.

% Let  $\{\tau_i\}_{i=1}^{n-1}$ be an   orthonormal basis for  $T_{x}(\partial\Omega)$. Then
%   $\nabla\phi^i_j=\sum_{k=1}^{n-1}\left(\frac{\partial\phi^i_j}{\partial\tau_k}\right)\tau_k$, thus |\nabla\phi^i_j|^2=\sum_{k=1}^{n-1}\left(\frac{\partial\phi^i_j}{\partial\tau_k}\right)^2$. Substituting this expression in  (\ref{rpd}), we obtain
%$$
%\frac{1}{d_{\sigma_1}}\sum_{j=1}^{d_{\sigma_1}}\sum_{k=1}^{n-1}\left(\frac{\partial\phi^1_j}{\partial\tau_k}\right)^2-\lambda(\partial\phi^1_j)^2=\frac{1}{d_{\sigma_2}} \sum_{j=1}^{d_{\sigma_2}}\sum_{k=1}^{n-1}\left(\frac{\partial\phi^2_j}{\partial\tau_k}\right)^2-\lambda(\phi^2_j)^2. $$

Using   (\ref{rpd}), it follows that  
\begin{equation}
\label{p1}
\frac{1}{d_{\sigma_1}}\sum_{j=1}^{d_{\sigma_1}}(\phi^1_j)^2=\frac{1}{d_{\sigma_2}}\sum_{j=1}^{d_{\sigma_2}}(\phi^2_j)^2.
\end{equation}
 Now, if   $d_{\sigma_i}=2$ for $i=1,2$, define
$$
F(x)=(\phi^1_1,\cdots, \phi^1_{d_{\sigma_1}})= (\phi^1_1,\phi^1_2),
$$ 

$$
G(x)=(\phi^2_1, \cdots,\phi^2_{d_{\sigma_2}})= (\phi^2_1,\phi^2_2) .
$$
 If one of the	 $d_{\sigma_i}$ is equal to   $1$ we just put the two coordinates equal to
  $\phi_{1}^i$.

It follows from  (\ref{p1}) and  (\ref{rsd})  that  $|F|=|G|$ e $|\frac{\partial F}{\partial\tau}|=|\frac{\partial G}{\partial\tau}|$. and then, from Lemma \ref{lsed}, there is   an orthogonal transformation $T$ such that
 $F(x)=TG(x)$ in an  open set  $V$  on  $\partial\Omega$. Thus we have, in particular  $\phi_1^1=\alpha\phi^2_1+\beta\phi_2^2$ on  $V$ and 
	\begin{equation*}
				\left\{
							\begin{array}{llll}
									(\Delta+\lambda)(\phi^1_i-\alpha\phi^2_1-\beta\phi_2^2)=0 &\ \ in \ \ \Omega;\\
									\frac{\partial}{\partial N}(\phi^1_i-\alpha\phi^2_1-\beta\phi_2^2)=0 &\  \ on \ \ \partial\Omega;\\
									\phi^1_i-\alpha_i\phi^2_1-\beta_i\phi_2^2=0 & \ \                    on\ \ V\cap\partial\Omega .
						  \end{array}\right.
	\end{equation*}
	From  Cauchy  uniqueness theorem \ref{tuc}, $\phi^1_i=\alpha\phi^2_1+\beta\phi_2^2$ in  $\Omega$, which is a contradiction since
   $ M_{d_{\sigma_1}} \cap  M_{d_{\sigma_2}} = 0 $.

\end{proof}

\begin{coro}\label{tgaaed1ln}
If  $G$ is a finite subgroup of  $O(n)$  such that $d_{\sigma}\leq2$ for all   $\sigma\in\hat{G}$,
 then, for a residual set set of open  bounded connected   $\mathcal{C}^3$-regular and  $G$-symmetric domains
 the eigenvalues of the problem (\ref{eln}) are all $G$-simple. 
\end{coro}
\begin{proof}
Let
$$ {\cal{C}}_k=\{h\in Diff^2_{G}(\Omega) \, | \, \textrm{ all eigenvalues} \  \lambda \textrm{ of } (\ref{eln})  \textrm{ with }
 \,  \lambda < k \textrm{ are all } G-\textrm{simple} \}.$$
 It is enough to prove that ${\cal{C}}_k$ is open and dense. 
 The proof is completely analogous to the one of Corollary \ref{cp1}, using Theorem 
 \ref{tslned}  instead of Theorem \ref{tsalnme}.
\end{proof}

\begin{obs} \label{planarsimple}
  The results above give a complete answer in the particular case of 
   compact subgroups of the $O(2)$.  In fact, in this case, the irreducible representations must have dimension at most 2. This is well known, and also follows from corollary \ref{amr}, since the eigenvalues of  the  Neumann Laplacian (for example in the disk of $\mathbb{R}^2$) have multiplicity 1 or 2.  Thus, Corollary \ref{tgaaed1ln} applies in the case of finite groups. In the infinite case,  the only invariant subgroups are  $SO(2)$ and  $O(2)$ itself. But then  the only invariant regions  are the disks, for which the result is well known.
     \end{obs}

%\begin{coro} \label{planarsimple}
%If  $G$ is compact subgroup $O(2)$, então para um conjunto residual de regiões $\Omega$ do $\mathbb{R}^2$, abertas, limitadas $\mathcal{C}^2$-regulares e $G$-simétricas, todos os autovalores para o problema \ref{eln} são $G$-simples.
%\end{coro}
%
%\begin{proof}
%Basta mostrarmos que $O(2)$ não possui subgrupos com dimensão de representação irredutível maior que 2. Com efeito, sabemos que todos os autovalores para a equação de Laplace com condição de fornteira de Neumann são simples ou duplos. Pelo corolário \ref{amr} temos que a multiplicidade dos autovalores é um múltiplo das representações irredutíveis de $G$ , portanto devemos ter que $d_{\sigma}<3$ para todo $G$ finito. Além disso , os únicos subgrupos compactos de $O(2)$ são o $SO(2)$ e o próprio $O(2)$. Portanto  segue o resultado. 
%\end{proof}

 The next result shows that the eigenvalues associated to subspaces 
  $M_{\sigma}$ with  $d_{\sigma}=1$  are generically simple, that is, they can be separated from the eigenvalues in other symmetry spaces. In particular,   generically in the set of $G$-symmetric regions there is an  infinite number of simple eigenvalues for the
   Neumann Laplacian.

	\begin{teo}\label{tsanaed1}
	
	 Let   $G$ be a finite subgroup of  $O(n)$ 
 and  $\Omega\subset\mathbb{R}^n$ an open  bounded connected   $\mathcal{C}^3$-regular and  $G$-symmetric domain.
  Suppose  that  $d_{\sigma_1}=1$  and  $\lambda$ is the unique  eigenvalue for the problem  (\ref{eln}) restricted to the  symmetry spaces  $M_{\sigma_1}$  and  $M_{\sigma_2}$ in the interval  $(\lambda-\delta,\lambda+\delta)$. Suppose also that 
  the action of  $G$  in both  $ker(\Delta|_{M_{\sigma_1}}+\lambda)$  and   $ker(\Delta|_{M_{\sigma_2}}+\lambda)$ is
 irreducible.  Then $\lambda$  can be separated by small $G-$ symmetric perturbations of $\Omega$ in two eigenvalues one of which is simple. More precisely,  for any  $\epsilon>0$, there exists  $h\in Diff_{G}^{3}(\Omega)$, $||h-i_{\Omega}||_{\mathcal{C}^3}<\epsilon$ 
 and  $\delta>0$ such that  there are exactly two  eigenvalues  $\lambda_1(h), \lambda_2(h)$ 
  for the problem (\ref{elnp}) restricted to the space  $M_{\sigma_1}\oplus M_{\sigma_2}$ in the interval  $(\lambda-\delta,\lambda+\delta)$, with $\lambda_1(h)$ simple. In other words, the natural action of  $G$ on  $ker(h^*\Delta h^{*-1}|_{M_{\sigma_2}\oplus M_{\sigma_1}}+\lambda_1(h))$ and  $ker(h^*\Delta h^{*-1}|_{M_{\sigma_2}\oplus M_{\sigma_1}}+\lambda_2(h))$ is irreducible.
	 
	     \footnote{It is important to observe that from the fact the the action of $G$ on  $Ker(\Delta|_{M_{\sigma_1}}+\lambda)$
	      is simple it does not follow that the action in   $Ker(\Delta+\lambda)$ is also simple.}
	\end{teo}
\begin{proof}  Assuming that the eigenvalues cannot be separated and following the arguments in the proof of Theorem  \ref{tslned}, we obtain the  functions in $\mathbb{R}^{d_{\sigma_2}}$ 
$$
F(x)=\phi_1^1(1,.....,1)
$$
$$
G(x)=(\phi_1^2,....,\phi^2_{d_{\sigma_2}})
$$
satisfying the relations 
\begin{equation*}
\left\langle G(x),G(x)\right\rangle=\left\langle F(x),F(x)\right\rangle =d_{\sigma_2}(\phi_1^1)^2
\end{equation*}
  and  
\begin{equation}\label{idfv}
 \left\langle \frac{\partial G}{\partial\tau}(x),\frac{\partial G}{\partial\tau}(x)\right\rangle= \left\langle \frac{\partial F}{\partial\tau}(x),\frac{\partial   F}{\partial\tau}(x)\right\rangle=d_{\sigma_2}\left(\frac{\partial \phi_1^1}{\partial \tau}\right)^2,
\end{equation}
for any $x\in\partial\Omega$, and $\tau \in T_x(\partial \Omega)$  Denoting $(1,1,...,1)=\stackrel{\rightarrow}{1}$, we can write  
$$
F(x)=\phi_1^1 A(x)\stackrel{\rightarrow}{1},
$$
where $A(x)$ is an orthogonal linear transformation  $F$. Differentiating, we obtain
\begin{eqnarray}
& & \frac{\partial F}{\partial x_i}=\frac{\partial\phi_1^1}{\partial x_i}A(x)\stackrel{\rightarrow}{1}+\phi_1^1\frac{\partial }{\partial x_i}A(x)\stackrel{\rightarrow}{1}\nonumber 
\end{eqnarray}
 It follows from (\ref{idfv})  that
\begin{eqnarray}
& & 2\frac{\partial\phi_1^1}{\partial x_i}\phi_1^1\left\langle A(x)\stackrel{\rightarrow}{1},\frac{\partial }{\partial x_i}A(x)\stackrel{\rightarrow}{1}\right\rangle+(\phi_1^1)^2\Big|\frac{\partial }{\partial x_i}A(x)\stackrel{\rightarrow}{1}\Big|^2 =0.\nonumber
\end{eqnarray}
Note  that, since   $\left\langle A(x)\stackrel{\rightarrow}{1},A(x)\stackrel{\rightarrow}{1} \right\rangle =\left\langle\stackrel{\rightarrow}{1},\stackrel{\rightarrow}{1}\right\rangle$, it follows that
$$
(\phi_1^1)^2\Big|\frac{\partial }{\partial x_i}A(x)\stackrel{\rightarrow}{1}\Big|^2 =0,
$$
for  $i=1,2,...,n-1$.
 Since $\phi_1^1\neq0$ in a dense set of  $\partial\Omega$, it follows that  $\nabla_{\partial\Omega}(A(x)\stackrel{\rightarrow}{1})=0$ and, therefore $A(x)\stackrel{\rightarrow}{1}$ is constant $\partial\Omega$. This implies that $\phi_j^1=a_j\phi_1^2$  on $\partial\Omega$ which cannot occur, since  $\phi_j^1\notin M_{\sigma_2}$.
\end{proof}

\begin{coro}\label{ias}
		Suppose that  $G$  is a finite subgroup of  $O(n)$ and $d_{\sigma}=1$. Then, for a residual set set of open  bounded connected   $\mathcal{C}^3$-regular,   $G$-symmetric domains
 the eigenvalues of the problem (\ref{eln})  in  the symmetry space $M_{\sigma}$ are simple. 

\end{coro}

\begin{proof} 
Let
 \[ {\cal{C}}= \{h\in Diff_G^2(\Omega)|  \, \textrm{ the eigenvalues of the problem 
 (\ref{eln}) with eigenfunctions in }    M_{\sigma}  \textrm{ are all }  G_{\sigma}-\textrm{simple }\} 
 \]
 Openness follows from Theorem  \ref{tcaln} and density from Theorem
  \ref{tsanaed1} above.

\end{proof}

\section{Boundary operators and the method of rapidly oscillating functions}\label{OF}

 We show  here how   the ``Method of rapidly oscillating functions'', developed in
\cite{hp} can be used   to obtain  necessary conditions for the operators   $\Xi$,
% $\Pi$
  and  $\Phi$, defined in  (\ref{ofln1}),
%(\ref{ofln2})
and  (\ref{ofln3}) 
to be of finite range. 
 We start with an auxiliary result.
\begin{lema}\label{levf}
Suppose  $S$  is  a  $\mathcal{C}^1$ manifold; $A$ and $B\in L^2(S)$ with  compact
 support; $\theta$ is a  $\mathcal{C}^1$ real valued function on   $S$ with  $\nabla_{\partial\Omega}\theta\neq0$ in the union of the supports of  $A$ and  $B$; $E$ is a finite dimensional
 subspace of $L^{2}(S)$ and  $u(\omega)\in E$ for all large $\omega\in\mathbb{R}$ satisfying
$$
u(\omega)=A\cos(\omega\theta)+B\sin(\omega\theta)+o(1)\,\, \textrm{ in }
 \,\, L^2(S) 
$$
as  $\omega\rightarrow\infty$. Then $A=B=0$.
\end{lema}
\begin{proof}
  See  \cite{hp}.

\end{proof}

 We do the computations in detail for the operator $\Xi$;
 the computations for  
 % $\Pi$  
  $\Phi$ are completely analogous. 

 Recall that    $\Xi$ was    defined 
 in  (\ref{ofln1}) by
\begin{equation}\label{oxi}
 \Xi(\sigma)= \sum_{j=1}^{d_{\sigma}}\sigma\frac{\partial}{\partial N}{\cal{Q}}(\phi_j^1,\phi_{j}^2) + {\cal{Q}}(\phi_j^1,{\cal{C}}_j^2(\sigma))+ {\cal{Q}}(\phi_j^2,{\cal{C}}_j^1(\sigma))
\end{equation}

where  ${\cal{C}}_j^i$  are the solutions  $\dot{\phi}^i_j = \dot{\varphi}_{(i-1)d_{\sigma}+j}$ of  (\ref{eofcln}) as functions of 
 $\sigma$,

%\begin{eqnarray}
%& & M_j^i(\sigma)=\nabla_{\partial\Omega}\phi_j^i\cdot\nabla_{\partial\Omega}\sigma-\sigma\frac{\partial^2\phi_j^ i}{\partial N^2} \label{dom}
%& & e\nonumber\\
% \end{eqnarray}
 
\begin{eqnarray}
& & {\cal{Q}}(u,v)=\nabla v\cdot\nabla u-\lambda vu. \label{dq}
\end{eqnarray}
 We will show that

$$
\Xi(\gamma \cos(\omega\theta))=\omega\gamma \cos(\omega\theta)\sum_{j=1}^{d_{\sigma}}\frac{\partial\phi_j^1}{\partial\theta}\frac{\partial\phi_j^2}{\partial\theta} + O(\omega)
$$
as  $\omega\rightarrow\infty$. Here
 $\frac{\partial }{\partial\theta}=   
 \nabla_{\partial \Omega} \theta \cdot \nabla_{\partial \Omega}    $ 
% $\frac{\partial\phi_j^i}{\partial\theta}= \nabla_{\partial \Omega} \phi_j^i  \nabla_{\partial \Omega} \theta   $ 
is the derivative in the
direction of $\nabla_{\partial \Omega} \theta$.
 If  $\Xi$ is supposed to be of finite rank, we conclude from Lemma  \ref{levf}, that 
$$
\sum_{j=1}^{d_{\sigma}}\frac{\partial\phi_j^1}{\partial\theta}\frac{\partial\phi_j^2}{\partial\theta}=0 \,\,  \textrm{on} \,\, \partial\Omega.
$$

Following the method presented in \cite{hp} we search first formal solutions $u=e^{\omega S(x)}\sum_{k\geq 0}\frac{U_k(x)}{(2\omega)^k}$ of

\begin{equation*}
\left\{\begin{array}{lccc}
       (\Delta+\lambda)u=(2\omega)F  & \ \ in \ \ \Omega;\\
       \frac{\partial u}{\partial N}=2\omega G(x) &\ \  on \ \ \partial\Omega;  
       \end{array}
\right.
\end{equation*}
where $F(x)=e^{\omega S(x)}\sum_{k\geq 0}\frac{F_k(x)}{(2\omega)^k}$, $G(x)=e^{i\theta(x)}\sum_{k\geq 0}\frac{G_k(x)}{(2\omega)^k}$, $\theta|_{\partial\Omega}$ given, $S|_{\partial\Omega}=i\theta$, $Re\frac{\partial S}{\partial N}|_{\partial \Omega}>0$ and $F_k, G_k$ are smooth  functions with values in $\mathbb{C}$.

We choose  the complex-valued $S$ so $\nabla S\cdot\ \nabla S=0$ on a neighborhood of $\partial\Omega$ and the $U_k$ inductively, solving 
\begin{equation*}
\left\{\begin{array}{lccc}
       \Lambda U_k+(\Delta+\lambda)U_{k-1}=F_k  & \ \ in \ \ \Omega;\\
       \frac{\partial U_{k-1}}{\partial N}+\frac{1}{2}\frac{\partial S}{\partial N}U_k=G_k &\ \  on \ \ \partial\Omega;  
       \end{array}
\right.
\end{equation*}
with $U_{-1}=0$, where $\Lambda=\nabla S\cdot\nabla +\frac{1}{2}\Delta S$. 
 They are not ordinarily,  exact solutions, but we only need that $\nabla S\cdot\nabla S$ and the $\Lambda U_k + (\Delta+\lambda)U_{k-1}-F_k$ tend to zero rapidly as $x\rightarrow\partial\Omega$, which is shown in   \cite{hp} (for the Dirichlet case, but the argument also applies here).

Using the notation above, we have 
$$
{\cal{C}}_j^i(\sigma) = e^{i\omega\theta}U^{i,j}_0+ O(1).
$$
Thus
\begin{eqnarray}
(\nabla \phi_j^1\cdot\nabla -\lambda \phi_j^1){\cal{C}}_j^2(\sigma) & = & \nabla_{\partial\Omega}\phi_j^1\cdot\nabla_{\partial\Omega}(e^{i\omega\theta}U^{2,j}_0)-\lambda e^{i\omega\theta}U^{2,j}_0\phi_j^1 \nonumber\\
                                                                & = & e^{i\omega\theta}\left\{i\frac{\partial\phi_j^1}{\partial\theta}U^{2,j}_0\omega+{\cal{Q}}(\phi_j^1,U^{2,j}_0) \right\}\nonumber\\
                                                                & = & ie^{i\omega\theta}\frac{\partial\phi_j^1}{\partial\theta}U^{2,j}_0\omega+O(1)\nonumber,
\end{eqnarray}
that is 
\begin{equation*}
{\cal{Q}}(\phi_j^1,{\cal{C}}_j^2(\sigma))=i\omega e^{i\omega\theta}\frac{\partial\phi_j^1}{\partial\theta}U^{2,j}_0\omega+O(1).
\end{equation*}
Analogously
\begin{equation*}
{\cal{Q}}(\phi_j^2,{\cal{C}}_j^1(\sigma))=i\omega e^{i\omega\theta}\frac{\partial\phi_j^2}{\partial\theta}U^{1,j}_0\omega+O(1).
\end{equation*}
Therefore
\begin{eqnarray}
\Xi(\gamma e^{i\omega\theta}) & = & e^{i\omega\theta}\left\{\sum_{j=1}^{d_{\sigma}}-i\omega\left(\frac{\partial\phi_j^1}{\partial\theta}U^{2,j}_0+\frac{\partial\phi_j^2}{\partial\theta}U^{1,j}_0\right)+\gamma\frac{\partial}{\partial N}{\cal{Q}}(\phi_j^1,\phi_j^2)\right\}\nonumber.
\end{eqnarray}
We want to determine the term 
 $U^{i,j}_0$ in the formal solution. To this end, using the notation
  $ M_j^i(\sigma)=\nabla_{\partial\Omega}\phi_j^i\cdot\nabla_{\partial\Omega}\sigma-\sigma\frac{\partial^2\phi_j^ i}{\partial N^2}$, we write 
$$M_j^i(\gamma e^{i\omega\theta})
  =2\omega e^{i\omega\theta}\sum_{k\geq0}\frac{G_k}{(2\omega)^k}.
$$ 
We have

\begin{eqnarray}
M_j^i(\gamma e^{i\omega\theta}) & = & \nabla_{\partial\Omega}(\gamma e^{i\omega\theta})\cdot\nabla_{\partial\Omega}\phi^i_j-\gamma e^{i\omega\theta}\frac{\partial^2 \phi^i_j}{\partial N^2}\nonumber\\
                                & = & e^{i\omega\theta}\left(\nabla_{\partial\Omega}\gamma\cdot\nabla_{\partial\Omega}\phi^i_j+\omega i\gamma\frac{\partial\phi_j^i}{\partial\theta}-\gamma \frac{\partial^2 \phi^i_j}{\partial N^2}\phi^i_j\right)\nonumber\\ 
                                & = & 2\omega e^{i\omega\theta}\left(i\gamma\frac{1}{2}\frac{\partial\phi_j^i}{\partial\theta}+\frac{1}{2\omega}M_j^ i(\gamma)\right)\nonumber.
\end{eqnarray}
Therefore 
$$
G_0=i\gamma\frac{1}{2}\frac{\partial\phi_j^i}{\partial\theta},
\,\,G_1= M_j^ i(\gamma),
$$
and then  $U_0^{i,j}= i\gamma e^{i\omega\theta}\frac{1}{4}\frac{\partial\phi_j^i}{\partial\theta}$. Therefore
\begin{eqnarray}
\Xi(\gamma e^{i\omega\theta}) & = & \gamma e^{i\omega\theta}\omega\sum_{j=1}^{d_{\sigma}}\frac{\partial\phi_j^1}{\partial\theta}\frac{\partial\phi_j^2}{\partial\theta}+O(1)\nonumber.\nonumber
\end{eqnarray}
Observing that
$$
\Xi(\gamma \cos(\omega\theta))=\frac{1}{2}Re\left\{\Xi(\gamma e^{i\omega\theta})+\Xi(\gamma e^{-i\omega\theta})\right\},
$$
it follows that
\begin{eqnarray}
\Xi(\gamma \cos(\omega\theta)) & = & \omega \gamma \cos(\omega\theta)\sum_{j=1}^{d_{\sigma}}\frac{\partial\phi_j^1}{\partial\theta}\frac{\partial\phi_j^2}{\partial\theta}+O(1)\nonumber.
\end{eqnarray}
If  $\Xi$ is of finite range we obtain, from Lemma \ref{levf}
\begin{equation*}
\sum_{j=1}^{d_{\sigma}}\frac{\partial\phi_j^1}{\partial\theta}\frac{\partial\phi_j^2}{\partial\theta}=0\,\,em\,\, \partial\Omega.
\end{equation*}

%Os próximos dois lemas são fundamentais para podermos concluir que a propriedade acima é válida para todo $\tau\bot T_x(G(x))$, em particular se o grupo $G$ for finito a identidade é verdadeira para $\tau\in\partial\Omega$.
%\begin{lema}\label{gfi}
%Sejam $G$ subgrupo compacto de $O(n)$ tal que existe $x\in\mathbb{R}^n$ livre pela ação de $G$; $\Omega\subset\mathbb{R}^n$, aberto, limitado com fronteira regular e $G$-invariante. Se $u:\partial\Omega\rightarrow\mathbb{R}$ é suave e $G$-invariante então $\nabla_{\partial\Omega}u\bot G(x)$.
%\end{lema}
%\begin{proof}
%Seja $\gamma: \mathbb{R}\rightarrow G$ suave tal que $\frac{d}{dt}\gamma|_{t=0}x=\tau\in T_{x}(G(x))$. Note que $u(\gamma(t)x)=u(x)$, então
%\begin{eqnarray}
%\nabla_{\partial\Omega}u\cdot\tau & = & \frac{d}{dt}|_{t=0}u(\gamma(t)x)\nonumber\\
%                                  & = & \frac{d}{dt}|_{t=0}u(x)\nonumber\\
%                                  & = &0\nonumber
%\end{eqnarray}
%\end{proof}
%\begin{lema}\label{vtf}
%Seja $V_{\epsilon}$ é uma vizinhança do espaço normal a $G(x)$ no ponto $x\in\Omega$, $\Omega$ com a mesmas condições que no lema acima. Dada uma função contínua em $V_{\epsilon}$ é sempre possível estendê-la a uma função $G$-invariante em $\Omega$. Como consequência $\nabla u$ pode ser qualquer vetor normal a $G(x)$. 
%\end{lema}
%\begin{proof}
%Ver \cite{as}.
%\end{proof}

\begin{teo}\label{tofln1}
Let  $G$ be a compact subgroup  of  $O(n)$;   $\Omega$ an open, bounded,  connected
 $\mathcal{C}^3$-regular and e $G$-symmetric region. Suppose the natural action of $G$ on 
${\partial \Omega}$ has a free point $x$ and 
 $\{\phi_j^i\}_{j=1}^{d_{\sigma}}$, $i=1,2$
 are eigenfunctions for the  problem (\ref{eln}) belonging to the symmetry space $M_{\sigma}$, satisfying
$$
\sum_{j=1}^{d_{\sigma}}{\cal{Q}}(\phi_j^1,\phi_j^2)=0
$$
on $\partial\Omega$, where $\cal{Q}$ was given in  (\ref{dq}). If  the operator
 $\Xi$  given in  (\ref{oxi}) is of finite range, then 
\begin{equation*}
\sum_{j=1}^{d_{\sigma}}\frac{\partial\phi_j^1}{\partial\tau}\frac{\partial\phi_j^2}{\partial\tau}=0
\end{equation*}
 in a neighborhood  $V$ of $x$ in   $\partial\Omega$, for all  $\tau \bot T_x(G(x))$.
 In particular,if $G$ is finite,   this is true for any  $\tau \in T_x(\partial\Omega)$.
\end{teo}
\begin{proof}
Taking (\ref{levf}) into account, it remains only to show that $\nabla \theta$ can be any
chosen to be any unit vector  $\tau \bot T_x(G(x))$. But this is guaranteed by
 Lemma 10.3 of   \cite{as}.
\end{proof}

Similar arguments lead to similar results for the operator
%$\Pi$
 $\Phi$ 
% definidos em (\ref{ofln2}) e 
defined in (\ref{ofln3}).
 
%\begin{equation}\label{opi}
%\Pi(\sigma)= \sum_{j=k}^{d_{\sigma}}\sigma\frac{\partial}{\partial N}{\cal{Q}}(\phi_k,\phi_{k})-2{\cal{Q}}(\phi_k,{\cal{C}}(M_k(\sigma)))
%\end{equation}  

\begin{eqnarray}\label{oflncof1}
\Phi(\sigma) & = & \frac{1}{d_{\sigma_1}}\sum_{j=1}^{d_{\sigma_1}}\sigma\frac{\partial}{\partial N}{\cal{Q}}(\phi^1_k,\phi^1_{k})-2({\cal{Q}}(\phi^1_k,{\cal{C}}_k^1(\sigma))\nonumber\\
             &   &-\frac{1}{d_{\sigma_2}}\sum_{j=1}^{d_{\sigma_2}}\sigma\frac{\partial}{\partial N}{\cal{Q}}(\phi^2_k,\phi^2_{k})-2({\cal{Q}}(\phi^2_k,{\cal{C}}_k^2(\sigma)).
\end{eqnarray}

%\begin{teo}\label{tofln2}
%Sejam $G$ subgrupo compacto de $O(n)$ que possui um ponto livre;  $\Omega$ aberto, conexo, limitado, $\mathcal{C}^3$-regular e $G$-simétrica; $\{\phi_j\}_{j=1}^{d_{\sigma}}$, são autofunções para o problema (\ref{eln}), pertencente ao espaço $M_{\sigma}$ ,que satisfazem 
%$$
%\frac{1}{d_{\sigma_1}}\sum_{j=1}^{d_{\sigma}}{\cal{Q}}(\phi_j^1,\phi_j^1)=0
%$$
%em $\partial\Omega$, onde $\cal{Q}$ está definido em (\ref{dq}). Se o operador $\Pi$ definido em (\ref{opi}) é de posto finito, então 
%\begin{equation*}
%\frac{1}{d_{\sigma_1}}\sum_{j=1}^{d_{\sigma_1}}\left(\frac{\partial\phi_j}{\partial\tau}\right)^2=0
%\end{equation*}
%em uma vizinhança $V$ de $\partial\Omega$, para todo $\tau \bot T_x(G(x))$. Em particular, $\tau \in T_x(\partial\Omega)$, se $G$ é finito.
%\end{teo}

\begin{teo}\label{tofln3}
Let  $G$ be a compact subgroup  of  $O(n)$;   $\Omega$ an open, bounded,  connected
 $\mathcal{C}^3$-regular and e $G$-symmetric region. Suppose the natural action of $G$ on 
${\partial \Omega}$ has a free point $x$ and 
 $\{\phi_j^i\}_{j=1}^{d_{\sigma}}$, $i=1,2$
 are eigenfunctions for the  problem (\ref{eln}) belonging to the symmetry space $M_{\sigma}$, satisfying
$$
\frac{1}{d_{\sigma_1}}\sum_{j=1}^{d_{\sigma}}{\cal{Q}}(\phi_j^1,\phi_j^1)=\frac{1}{d_{\sigma_2}}\sum_{j=1}^{d_{\sigma_2}}{\cal{Q}}(\phi_j^2,\phi_j^2)
$$
on  $\partial\Omega$, where $\cal{Q}$ was given  (\ref{dq}). If the operator $\Phi$
 given in  (\ref{oflncof1}) is of finite range, then 
\begin{equation*}
\frac{1}{d_{\sigma_1}}\sum_{j=1}^{d_{\sigma_1}}\left(\frac{\partial\phi_j^1}{\partial\tau}\right)^2=\frac{1}{d_{\sigma_2}}\sum_{j=1}^{d_{\sigma_2}}\left(\frac{\partial\phi_j^2}{\partial\tau}\right)^2
\end{equation*}
 in a neighborhood  $V$ of $x$ in   $\partial\Omega$, for all $\tau \bot T_x(G(x))$.
 In particular, if $G$ is finite this is true for any  $\tau \in T_x(\partial\Omega)$.
\end{teo}

\vspace{1cm}
\noindent \textsc{Ant\^onio Luiz Pereira} \\
Instituto de Matem\'atica e Estat\'istica - Universidade de S\~ao Paulo, Rua do Mat\~ao, 1010 \\
Cidade Universit\'aria - CEP 05508-090 - S\~ao Paulo - SP - Brazil \\
\texttt{alpereir@ime.usp.br} \\

\vspace{.5cm}
\noindent \textsc{Marcus A. M. Marrocos} \\
Universidade Federal do Amazonas, Instituto de Ci\^encias Exatas-ICE, Departamento de Matem\'atica \\
Av. General Rodrigo Oct\'avio Jord\~ao Ramos, 3000, Campus Universit\'ario Coroado I \\
69077-070 - Manaus, AM - Brasil\\
\texttt{marcusmarrocos@gmail.com}

\end{document}